\documentclass{amsart}
\usepackage{amsrefs}
\usepackage{url}
\usepackage[colorlinks=true, linkcolor=black, citecolor=black]{hyperref}
\usepackage{amsmath, amssymb, amsthm}
\usepackage[english]{babel}
\usepackage{tikz-cd}

\newtheorem{thm}{Theorem}[section]
\newtheorem{cor}[thm]{Corollary}
\newtheorem{prop}[thm]{Proposition}
\newtheorem{lem}[thm]{Lemma}

\newtheorem{hyp}[thm]{Assumption}
\newtheorem{quest}[thm]{Question}

\title[Random products of birational maps]{Random products of birational maps: equidistribution of preimages of curves}
\author{Arnaud Nerrière}
\date{}
\address{Université Bourgogne Europe, CNRS, IMB UMR 5584, 21000 Dijon, France}
\email{arnaud.nerriere@u-bourgogne.fr}

\begin{document}

\begin{abstract}
    We consider iterated preimages of curves by random products of birational transformations of the plane. Following a recent work of Diller and Roeder, we study 
    the action of the Cremona group on the inverse limit of the spaces of currents in all models of the plane. We show equidistribution for generic finitely supported random walks.
\end{abstract}

\maketitle
\tableofcontents

\section{Introduction}

Let $\mu$ be a finitely supported probability measure on the Cremona group $\mathrm{Bir}(\mathbb{P}^2_\mathbb{C})$ of birational transformations of the plane. We define $\Sigma_\mu$ to be the support of $\mu$ and we set $\Omega:=\Sigma_\mu^\mathbb{N}$. We denote by $f^n_\omega:=f_n \cdots f_1$ the left product of the first $n$ terms of $\omega\in\Omega$ for any $n \geq 1$. Given a curve $\mathcal{C } \subset \mathbb{P}^2$, we are interested in the asymptotic behaviour of the sequence of preimages of $\mathcal{C}$ by $f^n_\omega$. This question is usually understood in the weak sense of currents, that is, we consider the current of integration on $\mathcal{C}$ and we want to know if the sequence of normalized pullbacks $$\frac{1}{\deg(f^n_\omega)}(f^n_\omega)^*\mathcal{C}$$ converges to a positive closed current of bidegree $(1,1)$ on $\mathbb{P}^2$. More generally, we consider the normalized pullbacks of any positive closed $(1,1)$ current on $\mathbb{P}^2$. This is the counterpart in the random setting to the classical equidistribution problem in holomorphic dynamics. 

This question was adressed by Cantat--Dujardin \cite{CD1} and by Filip--Tosatti \cite{FT} for random products of automorphisms on compact Kähler surfaces. In the affine setting, the author \cite{N} proved the existence of Green functions associated to random products of automorphisms of the affine plane. In these two cases, the Green currents (or functions) play an important role for the study of random orbits, in particular in the classification of stationary measures. We consider here the equidistribution problem in the more general framework of products of \emph{iid} birational transformations of $\mathbb{P}^2$.

One major issue comes from the lack of algebraic stability: the degree of $f^n_\omega$ is not equal in general to the product of the degrees of the $f_i$. Equivalently, $f^n_\omega$ may contract a curve to an indeterminacy point of $f_{n+1}$. Unlike for the case of a single birational map \cite{DiF}, we cannot expect to birationally conjugate the dynamics so that it becomes algebraically stable. 

Still, we are able to prove equidistribution of pullbacks of any current for a generic, finitely supported measure $\mu$ on $\mathrm{Bir}(\mathbb{P}^2)$. We denote by $\mathrm{Bir}_d(\mathbb{P}^2)$ the set of birational maps of degree $d$, for any $d\geq 1$. Recall that $\mathrm{Bir}_d(\mathbb{P}^2)$ is a quasiprojective variety \cite[Chapter 5]{La2} and note that $\mathrm{Bir}_1(\mathbb{P}^2)=\mathrm{PGL}_3(\mathbb{C})$. 

\begin{thm}\label{thm1}
    For any $r \geq 2$ and $d_1,\dots, d_r \in \mathbb{N}^\times$ with $\min d_i \geq 2$, there exists a Zariski dense subset $V \subset \mathrm{Bir}_{d_1}(\mathbb{P}^2)\times \cdots \times \mathrm{Bir}_{d_r}(\mathbb{P}^2) $
    such that for any probability measure $\mu$ supported on $\{f_1^{\pm1}, \dots, f_r^{\pm 1}\}$ with $(f_1,\dots,f_r) \in V$, then the following hold.    
    \begin{enumerate}
        \item For $\mu^\mathbb{N}$-almost every $\omega = (f_n)_{n \geq 1}\in \Omega$, there exists a positive closed $(1,1)$ current $T_\omega$ of mass one on $\mathbb{P}^2$ which does not charge curves, such that
        $$\frac{1}{\deg(f^n_\omega)} (f^n_\omega)^*S \longrightarrow T_\omega $$
        for every positive closed $(1,1)$ current $S$ of mass one on $\mathbb{P}^2$.
        \item Moreover, $T_\omega \neq T_{\omega'}$ for $\mu^\mathbb{N} \otimes \mu^\mathbb{N} $-almost every $(\omega,\omega') \in \Omega \times \Omega$.
    \end{enumerate}
\end{thm}

Compare Theorem \ref{thm1}.(2) to \cite[Theorem 6.11]{CD1} and \cite[Theorem 1.3]{N}. This property that limit currents depend non-trivially on the itinerary $\omega$ is an important ingredient for the classification of stationary measures \cite{Roda}.

Remark that we have equidistribution of pullbacks of all currents. One cannot expect this propery to hold for every $\mu$. For example, if $\mu$ is supported on $\text{Aut}(\mathbb{A}^2_\mathbb{C})$, then the current of integration on the line at infinity is invariant.

\subsection{Related results: Equidistribution in random holomorphic dynamics}
The importance of the equidistribution problem in random dynamics comes from the following observation, due to Cantat and Dujardin \cite{CD1}. Consider a random dynamical system $(X,\mu)$ where $X$ is a projective surface and $\Sigma_\mu$ is a finite subset of the automorphism group of $X$. Following Breiman law of large numbers, an important part of the study of the random dynamical system is the classification of $\mu$-stationary measures, that are the Borel probability measures $\nu$ on $X$ that satisfy $\sum \mu(f) f_ *\nu =\nu.$ If a $\mu$-stationary measure $\nu$ is hyperbolic, that is, if its Lyapunov exponents satisfy $\lambda^+>0>\lambda^-$, then the stable manifolds $W^s(\omega,x)$ are almost surely complex one-dimensional submanifolds that are biholomorphic to $\mathbb{C}$. A key ingredient in order to apply the deep results of Brown--Rodriguez Hertz \cite{BRH}, Brown--Eskin--Filip--Rodriguez Hertz \cite{BEFRH} and Roda \cite{Roda} is to show that stable manifolds depend non-trivially on the itinerary $\omega\in \Omega$. To prove this, Cantat and Dujardin first show that the Ahlfors currents associated to the stable manifolds $W^s(\omega,x)$ are equal to the currents obtained by limit of normalized pullbacks of any Kähler form by $f^n_\omega$. Then, they show that these currents depend non-trivially on the itinerary using boundary properties of the induced random walk on the cohomology space $H^{1,1}(X,\mathbb{R})$. 

The author \cite{N} extended these ideas to random dynamical systems $(\mathbb{C}^2,\mu)$ on the affine plane with $\mu$ a finitely supported measure on $\text{Aut}(\mathbb{A}^2_\mathbb{C})$. In this case, the Green functions $$G_\omega:=\lim \frac{1}{\deg(f^n_\omega)} \log^+\Vert f^n_\omega \Vert$$ are well-defined almost surely under the hypothesis that the group generated by $\Sigma_\mu$ contains only loxodromic elements. As above, the Green functions depend non-trivially on the itinerary $\omega$ and this implies rigidity results for stationary measures. The Green functions are related to the equidistribution problem by the following: for $\mu^\mathbb{N}$-almost every $\omega \in \Omega$,  $$\frac{1}{\deg(f^n_\omega)}(f^n_\omega)^*\kappa \longrightarrow  dd^c G_\omega$$
where $\kappa$ denotes any Kähler form on $\mathbb{P}^2$.

\subsection{Related results: Equidistribution in deterministic holomorphic dynamics}
The classical equidistribution problem is of crucial importance in holomorphic dynamics. Brolin's theorem \cite{Br} asserts that for a polynomial map $f: \mathbb{C} \rightarrow \mathbb{C}$ of degree $\deg(f)\geq 2$, the preimages of every point (except at most two) equidistributes towards the equilibrium measure of the Julia set of $f$. Brolin's theorem was generalized to rational maps $f: \mathbb{P}^1 \rightarrow \mathbb{P}^1$ separately by Lyubich \cite{Ly} and Freire--Lopes--Mañe \cite{FLM}. In higher dimensions, pioneering results were obtained by Bedford--Smillie \cite{BS} for Hénon maps and by Fornaess--Sibony \cite{FS} for holomorphic maps of $\mathbb{P}^2$. 

We say that a rational map $f : \mathbb{P}^2 \dashrightarrow \mathbb{P}^2$ is algebraically stable if $\deg(f^n)=\deg(f)^n$ for all $n \geq 1$. Under this hypothesis, Sibony \cite{Si} proved that if $f$ has degree at least two, then, for almost every curve $\mathcal{C} \subset \mathbb{P}^2$, we have equidistribution of the normalized pullbacks of $\mathcal{C}$ by $f^n$ to the so-called Green current of $f$. 

One can also study the action of $f$ by pullback on the cohomology space of any surface $X$ that is obtained by a finite number of blow-ups of $\mathbb{P}^2$. We say that $f$ is algebraically stable on $X$ if $(f^n)^*=(f^*)^n$ on $H^{1,1}(X)$ for all $n\geq 1$ . An important amount of work was made to understand which rational maps admit stable models and to prove equidistribution in this case \cites{DiF, DDG, Birkett}.

Beyond the algebraically stable case, the question of existence of a Green current remains open in general. 
Recently, Diller and Roeder \cite{DR} obtained an equidistribution theorem for toric rational maps that do not admit any stable model. Building on this equidistribution result, they prove the existence of a rational map with transcendental topological entropy \cite{DR2}.

\subsection{Ingredients of the proof}
Following Diller--Roeder \cite{DR}, we define the space of Weil currents $W(\mathbb{P}^2)$ to be the inverse limit of spaces of currents on all surfaces over $\mathbb{P}^2$. See also the work of Xia \cite{Xia}. The space $W(\mathbb{P}^2)$ admits an action of $\mathrm{Bir}(\mathbb{P}^2)$ by pullback and we adress the equidistribution problem in the sense of Weil currents. The analogue space of Weil classes $\overline{W}(\mathbb{P}^2)$ was introduced in the pioneering works of Cantat \cite{Cantat} and Boucksom--Favre--Jonsson \cite{BFJ}. By a theorem of Maher and Tiozzo \cite{MT2}, we have convergence at the level of cohomology: for $\mu^\mathbb{N}$-almost every $\omega \in \Omega$, $$\frac{1}{\deg(f^n_\omega)} (f^n_\omega)^*[L] \longrightarrow \theta_\omega \in \overline{W}(\mathbb{P}^2) $$
where $[L]$ denotes the (Weil) class of a line $L\subset \mathbb{P}^2$. We then have to pass from cohomology to currents. This strategy is classical when one works with a single model $X$ on which the dynamics is algebraically stable. See for example \cites{DDG, DiF}.

Working on all models, we consider the following strategy: consider a Weil current $S \in W(\mathbb{P}^2)$ that is cohomologous to $[L]$. Using compactness properties of the space $W(\mathbb{P}^2)$, it suffices to show that the sequence of normalized pullbacks of $S$ admits at most one limit point. Suppose that all these limit currents do not charge curves in any model. Thus, the limit currents are obtained by taking the trivial extension of their realization on $\mathbb{P}^2$. In this case, the Lelong numbers of these currents can be related to the coefficients of the limit class $\theta_\omega$ in the canonical basis of $\overline{W}(\mathbb{P}^2)$. Then, using weak volume estimates and Skoda's uniform integrability theorem, we can prove that the sequence of normalized pullbacks has at most one limit point.

Unfortunately, we do not know in general if such limit currents do not charge curves. We are able to prove it when $\mu$ is  generic and finitely supported. Note that in this case we have nice properties of the sequence of degrees and we would be able to prove equidistribution by working only on $\mathbb{P}^2$. We prefer to work on all models as it allows us to show that limit currents depend non-trivally on the itinerary.  

We also hope that these ideas may be useful in other contexts. For example, if $f: \mathbb{P}^2 \dashrightarrow \mathbb{P}^2 $ is a rational map with small topological degree $\lambda_1 >d_{top}$, then, by Boucksom--Favre--Jonsson \cite{BFJ}, the convergence at the level of classes is known. Thus, our method gives the following: if the limit of any subsequence of $T_n:=\frac{1}{\lambda_1^n}(f^n)^*S$ do not charge curves in any model, then $(T_n)$ converges. See the discussion at the end of Section \ref{section6}. 

\begin{quest}
   Consider a rational map $f: \mathbb{P}^2 \dashrightarrow \mathbb{P}^2$ with $\lambda_1>d_{top}$. Let $\kappa$ be the Fubini-Study form on $\mathbb{P}^2$, and consider the sequence of normalized pullbacks $$T_n:=\frac{1}{\deg(f^n)} (f^n)^*\kappa.$$
   If $T$ is a limit of a subsequence of $(T_n)$ (in the sense of Weil currents), can we prove that $T$ do not charge curves in any model ?
\end{quest}

\subsection{Organization of the paper}
Section \ref{section2} is devoted to the definition of the spaces of Weil and Cartier currents, that are inverse and direct limits of spaces of currents on all models of $\mathbb{P}^2$. Then, we gather useful results on the behaviour of the random walk on the space of Weil classes, and on the growth of sequence of degrees in Section \ref{section3}. We prove that limit currents do not charge curves for a generic $\mu$ in Section \ref{section4} and Section \ref{section 5}. The proof of Theorem \ref{thm1} occupies Section \ref{section6}.

\subsection{Acknowledgment}
This work was partially funded by the French National Research Agency under the project DynAtrois (ANR-24-CE40-1163), by the EIPHI Graduate School (contract ANR-17-EURE-0002), and by the Région Bourgogne-Franche-Comté. The author would like to thank Marc Abboud, Serge Cantat, Romain Dujardin, Christophe Dupont, Charles Favre, Stéphane Lamy, Andres Quintero Santander and Roland Roeder for interesting discussions related to this paper, and his Ph.D. advisors Johan Taflin and Michele Triestino for their support.

\section{Weil and Cartier currents}\label{section2}
In this section, we define and study the spaces of currents (and classes) in all surfaces over $\mathbb{P}^2$. Our main references are \cite[Chapter 8]{GZ} for basics on pluripotential theory on projective surfaces and \cite{La2} for basics on birational transformations of surfaces. See also \cites{Cantat, DR, BFJ}.

\subsection{Models}
A model $(X,\pi)$ of $\mathbb{P}^2$ is the data of a smooth projective surface $X$ and a birational morphism $\pi : X \rightarrow \mathbb{P}^2$. We say that $(X_1,\pi_1)$ dominates $(X_2, \pi_2)$ if the induced birational map $\pi_2^{-1}\pi_1$ is a morphism, and we say that the two models are equivalent if the induced map is an isomorphism.

\[
\begin{tikzcd}
    X_1   \arrow[rd, "\pi_1", swap] \arrow[rr, "\pi_2^{-1}\pi_1", dashed]  &  &X_2 \arrow[ld, "\pi_2"] \\
    &\mathbb{P}^2&
\end{tikzcd}
\]
We will always identify two equivalent models. A birational morphism $\pi: X \rightarrow \mathbb{P}^2$ can be decomposed as a product of blow-ups, that is, we can write $\pi=\pi_1 \cdots \pi_n$, where each $\pi_i$ is the blow-up of a point in $X_{i-1}$:
$$X=X_n \overset{\pi_n}{\longrightarrow} X_{n-1} \longrightarrow \dots \; X_1 \overset{\pi_1}{\longrightarrow} X_0:=\mathbb{P}^2. $$
If $(X_1,\pi_1)$ and $(X_2,\pi_2)$ are two models, then there exists a third one $(Y,\pi_Y)=(Y,\pi_1 \alpha)=(Y,\pi_2\beta)$ that dominates both: 
\[ 
\begin{tikzcd}
    & \arrow[ld, "\alpha", swap] Y \arrow[dd, "\pi_Y"] \arrow[rd, "\beta"] &\\
    X_1 \arrow[rd,"\pi_1", swap] & & X_2 \arrow[ld, "\pi_2"]\\
    & \mathbb{P}^2. &
\end{tikzcd}
\]

\subsection{Weil currents}
Consider a model $(X,\pi)$ of $\mathbb{P}^2$. We denote by $\mathcal{D}^+(X,\pi)$ the set of positive closed currents of bidegree $(1,1)$ on $X$. Note that this set depends only on $X$, but we keep the reference to the birational morphism $\pi$ for further use. We denote also by $\mathcal{D}(X,\pi)$ the space of currents that can be written as a difference of two elements of $\mathcal{D}^+(X,\pi)$. A sequence $(T_n)$ in $\mathcal{D}(X,\pi)$ converges to a current $T=T^+-T^-$ if and only if we can write $T_n=T_n^+-T_n^-$ with $T_n^+, T_n^- \in \mathcal{D}^+(X,\pi)$ and $T_n^+ \rightarrow T^+$, $T_n^- \rightarrow T^-$ for the weak topology of currents.

If $(X_1,\pi_1)$ dominates $(X_2,\pi_2)$, then we write $\pi_1=\pi_2\phi_{X_1,X_2}$ with a birational morphism $\phi_{X_1,X_2}: X_1\rightarrow X_2$. We can define the pushforward map 
$$(\phi_{X_1,X_2})_*:\mathcal{D}(X_1,\pi_1) \longrightarrow \mathcal{D}(X_2,\pi_2).$$
Remark that $(\phi_{X_1,X_2})_*$ preserves positivity, and if $(X_3,\pi_3)\leq (X_2, \pi_2) \leq (X_1,\pi_1)$, then we have $$(\phi_{X_1,X_3})_*=(\phi_{X_2,X_3})_* (\phi_{X_1,X_2})_*.$$

Thus, the spaces $\mathcal{D}(X,\pi)$ with associated pushforward maps form a projective system and we can define the space of \textbf{Weil currents} by $$W(\mathbb{P}^2):=\lim_{\longleftarrow} \mathcal{D}(X,\pi).$$
More concretely, a Weil current $T\in W(\mathbb{P}^2)$ is a family of currents $T=(T_{X,\pi})$ indexed by all models of $\mathbb{P}^2$, which is compatible with pushforward: if $(X_1,\pi_1)$ dominates $(X_2,\pi_2)$, then we have $$(\phi_{X_1,X_2})_ * T_{X_1,\pi_1}=T_{X_2,\pi_2}.$$
We say that $T_{X,\pi} $ is the \textbf{realization} of $T$ on $(X,\pi)$. We define also $$W^+(\mathbb{P}^2):=\lim_{\longleftarrow} \mathcal{D}^+(X,\pi)$$ to be the convex cone of positive Weil currents. We endow $W(\mathbb{P}^2)$ with the product topology. 

\subsection{Cartier currents}
Consider a birational morphism $\phi:X_1 \rightarrow X_2$  between two models of $\mathbb{P}^2$. We can define the pullback of a positive closed current on $X_2$ of bidegree $(1,1)$ by $\phi$ as follows. We  write $T=\eta + dd^cu$ with $\eta$ a smooth $(1,1)$ form and $u$ a quasi-psh function, and we set $$\phi^* T:=\phi^*\eta +dd^c(u \circ \phi).$$
For any $T=T^+-T^- \in \mathcal{D}(X,\pi),$ we set $\phi^*T:=\phi^*T^+ -\phi^*T^-.$ If $(X_1,\pi_1) $ dominates  $(X_2,\pi_2)$, then we write $\pi_1=\pi_2 \phi_{X_1,X_2}$, and  we have a well-defined pullback map
$$\phi_{X_1,X_2}^*: \mathcal{D}(X_2, \pi_2) \longrightarrow \mathcal{D}(X_1,\pi_1),$$
which preserves positivity. We have the following compatibility of the pullback maps: if $(X_3,\pi_3)\leq (X_2, \pi_2) \leq (X_1,\pi_1)$, then we have $$(\phi_{X_1,X_3})^*=(\phi_{X_1,X_2})^* (\phi_{X_2,X_3})^*.$$

Thus, the spaces $\mathcal{D}(X,\pi)$ with associated pullback maps is a direct system and we can define the space of \textbf{Cartier currents} by 
$$C(\mathbb{P}^2):=\lim_{\longrightarrow} \mathcal{D}(X,\pi).$$
More concretely, a Cartier current is the data of a model $(X,\pi)$ and a current in $\mathcal{D}(X,\pi)$ up to the following equivalence: we identify $T_1 \in \mathcal{D}(X_1,\pi_1)$ and $T_2 \in \mathcal{D}(X_2,\pi_2)$ if we have $\alpha^*T_1=\beta^*T_2$ on any model $Y$ that dominates $X_1$ and $X_2$.
\[ 
\begin{tikzcd}
    & \arrow[ld, "\alpha", swap] Y \arrow[dd, "\pi_Y"] \arrow[rd, "\beta"] &\\
    X_1 \arrow[rd,"\pi_1", swap] & & X_2 \arrow[ld, "\pi_2"]\\
    & \mathbb{P}^2 &
\end{tikzcd}
\]
Similarly, we define $$C^+(\mathbb{P}^2):=\lim_{\longrightarrow} \mathcal{D}^+(X,\pi)$$
to be the convex cone of positive Cartier currents.
We say that a Cartier current $T$ can be realized on a model $(X,\pi)$ if $T$ is equal in $C(\mathbb{P}^2)$ to a current in $\mathcal{D}(X,\pi)$. 

We can embed $C(\mathbb{P}^2)$ as a subspace of $W(\mathbb{P}^2)$: if $T$ is a Cartier current that can be realized on some model $(X_1,\pi_1)$, then we define the realization of $T$ on $(X_2,\pi_2)$ by $$T_{X_2,\pi_2}:=\beta_*\alpha^*T$$
with $\alpha,\beta$ and $Y$ as above.
We can check that it gives a Weil current, and that it does not depend on the choice of $X_1$ and $Y$. We endow $C(\mathbb{P}^2) $ with the subspace topology, and we remark that $C(\mathbb{P}^2)$ is dense in $W(\mathbb{P}^2)$.

\subsection{Strict transform}
Consider a positive closed current  $T_{\mathbb{P}^2}\in \mathcal{D}^+(\mathbb{P}^2)$ and a model $(X,\pi)$ of $\mathbb{P}^2$. We denote by $E$ the exceptional divisor of $\pi$. By Skoda--El Mir theorem \cite[Theorem 7.26]{GZ}, the pullback of $T_{\mathbb{P}^2}$ by the restriction of $\pi$ to $X\setminus E$ can be extended to a positive closed current on $X$. We call it the trivial extension, or the strict transform of $T_{\mathbb{P}^2}$ on $(X,\pi)$. We define the strict transform of a difference of positive closed currents to be the difference of the strict transforms. Finally, we define the \textbf{strict transform} of the current $T_{\mathbb{P}^2} \in \mathcal{D}(\mathbb{P}^2)$ to be the Weil current $\widetilde{T_{\mathbb{P}^2}} \in W(\mathbb{P}^2)$ given by the strict transform of $T_{\mathbb{P}^2}$ in any model. 

If $T=(T_{X,\pi})\in W(\mathbb{P}^2)$ is a Weil current, then by Siu's theorem \cite[Theorem 7.27]{GZ}, the realization of $T$ on $(X,\pi)$ differs from the strict transform of its realization on $\mathbb{P}^2$ by a current of integration  on the exceptional divisor of $(X,\pi)$. Thus, $T$ is a strict transform if and only if $T_{X,\pi}$ does not charge the exceptional divisor in any model $(X,\pi)$.

Note that a current $T_{\mathbb{P}^2}\in \mathcal{D}(\mathbb{P}^2)$ also defines a Cartier current which differs in general from the strict transform $\widetilde{T_{\mathbb{P}^2}}$. For example, if $T_{\mathbb{P}^2}=L$ is the current of integration on a line $L$, and if $(X,\pi)$ is the blow-up of a point in $L$, then we have $$\pi^*L=\widetilde{L}+E$$
where $E$ is the current of integration on the exceptional divisor of $\pi$. Thus, the realization on $(X,\pi)$ of the Cartier current defined by $L$ is not equal to $\widetilde{L}$. 

\subsection{Weil and Cartier classes}
We have a surjective map $\mathcal{D}(X,\pi) \longrightarrow H^{1,1}_\mathbb{R}(X,\pi)$ for any model,
which associates to a closed current $T_{X,\pi}$ its cohomology class $[T_{X,\pi}]$. By the $dd^c$-lemma \cite[Lemma 7.31]{GZ}, two currents $T, S$ are cohomologous if and only if we can write $$T-S=dd^cu$$
where $u$ is a $L^1$ function on $X$, given locally as a difference of plurisubharmonic functions. The pullback and the pushforward maps preserve the spaces of exact currents.
Thus, we can define the space of \textbf{Weil classes} $$\overline{W}(\mathbb{P}^2):=\lim_{\longleftarrow} H^{1,1}_\mathbb{R}(X,\pi),$$ 
and the space of \textbf{Cartier classes} 
$$\overline{C}(\mathbb{P}^2):=\lim_{\longrightarrow} H^{1,1}_\mathbb{R}(X,\pi).$$
We endow $\overline{W}(\mathbb{P}^2)$ with the product topology. We have continuous projections $W(\mathbb{P}^2) \rightarrow \overline{W}(\mathbb{P}^2)$ and $C(\mathbb{P}^2) \rightarrow \overline{C}(\mathbb{P}^2)$ which associate to a Weil (or Cartier) current its cohomology class. 

\subsection{Notation}\label{section2.6}
Consider a $\mathbb{R}$-divisor $D$ in a model $(X,\pi)$. We use the same notation $D$ for the current of integration on $D$, and also for the Cartier current that it defines. We denote by $[D]\in \overline{C}(\mathbb{P}^2)$ its cohomology class. The only important distinction is for the strict transform: we always keep the notation $\widetilde{D}$, since it differs in general from $D$.

\subsection{Canonical basis}

Consider a triple $(X,\pi,p) $ where $(X,\pi)$ is a model of $\mathbb{P}^2$ and $p$ is a point of $X$. Two such triples $(X_1, \pi_1, p_1)$ and $(X_2,\pi_2,p_2)$ are said to be equivalent if $\pi_2^{-1}\pi_1$ is a local isomorphism sending $p_1$ to $p_2$. We denote by $\mathcal{B}(\mathbb{P}^2)$ the quotient space. The order of a point $p \in \mathcal{B}(\mathbb{P}^2)$ is the minimal number of blowups in $(X,\pi)$ such that $p\in X$. See \cite[Section 6]{La2}.

Given $p\in \mathcal{B}(\mathbb{P}^2)$, we take a model $(X,\pi)$ containing $p$, and we denote by $E_p$ the Cartier current defined by the exceptional divisor of the blow-up of $p$. We can check that it does not depend on the choice of $(X,\pi)$.

Let $\pi: X \rightarrow \mathbb{P}^2$ be a model. We decompose $\pi$ as a product of blow-ups of the points $p_1, \dots ,p_n \in \mathcal{B}(\mathbb{P}^2)$. 
Then, we have  $$H^{1,1}_\mathbb{R}(X,\pi)=\text{Vect}_\mathbb{R}([L], [E_{p_1}], \dots ,[E_{p_n}]).$$
Here, $E_{p_i}$ denotes the realization of the Cartier current $E_{p_i}$ on $(X,\pi)$, which is the pullback on $X$ of the exceptional divisor obtained by the blow-up of $p_i$. We obtain the following: $$\overline{C}(\mathbb{P}^2)=\mathbb{R} [L]\oplus \bigoplus_{p \in \mathcal{B}(\mathbb{P}^2)} \mathbb{R}[E_p].$$
We say that $([L],[E_p])_{p \in \mathcal{B}(\mathbb{P}^2)}$ is the canonical basis of $\overline{C}(\mathbb{P}^2)$. We have also 
$$\overline{W}(\mathbb{P}^2)=\mathbb{R} [L]\oplus \prod_{p \in \mathcal{B}(\mathbb{P}^2)} \mathbb{R}[E_p].$$
We write $$c=a_L [L]+\sum_{p\in \mathcal{B}(\mathbb{P}^2)} a_p [E_p]$$
for a Weil class $c\in \overline{W}(\mathbb{P}^2)$. This means that the realization of $c$ on $(X,\pi)$ is given by $$c_{X,\pi}=a_L [L] +\sum_{p \in \mathrm{base}(\pi)} a_p[E_p]$$
where the \textbf{base points} of $\pi$ are the points of $\mathcal{B}(\mathbb{P}^2)$ blown-up by $\pi$. 

Note that Siu's theorem implies a similar result for currents, that is
$$C(\mathbb{P}^2)=\mathcal{D}(\mathbb{P}^2) \oplus \bigoplus_{p \in \mathcal{B}(\mathbb{P}^2)} \mathbb{R}E_p $$
and 
$$W(\mathbb{P}^2)=\mathcal{D}(\mathbb{P}^2) \oplus \prod_{p \in \mathcal{B}(\mathbb{P}^2)} \mathbb{R}E_p. $$

\subsection{Intersection form}
Each cohomology space $H^{1,1}_\mathbb{R}(X,\pi)$ is endowed with an intersection form $\langle , \rangle_{X,\pi}$ which is non-degenerate and of Minkowski type. Moreover, they satisfy the following compatibility condition \cite[Proposition 1.109]{La2}: $$\langle \phi^*c,d\rangle_{X_1,\pi_1}=\langle c,\phi_*d\rangle_{X_2,\pi_2}$$
when $\pi_1=\pi_2\phi$ with $\phi$ a birational morphism. As a special case, we remark that $$\langle \phi^*c, [E_p] \rangle =0$$ when $E_p$ is contracted by $\phi$. We thus have a well-defined intersection form on $\overline{C}(\mathbb{P}^2)$. The canonical basis is an orthonormal basis for $\langle ,\rangle$ and we have $$c^2=a_L^2 - \sum_{p \in \mathcal{B}(\mathbb{P}^2)} a_p^2$$
for any Cartier class $c\in \overline{C}(\mathbb{P}^2)$. 

We say that a class $c \in H^{1,1}(X,\pi)$ is \textbf{nef} if $\langle c, [D] \rangle \geq 0 $ for every effective divisor $D$ on $X$. Similarly, we say that a Weil class is nef it its realization in every model is nef. We have the following push-pull formula.

\begin{lem}\label{lem6}
    If $\pi: X \rightarrow Y$ is a birational morphism between two models of $\mathbb{P}^2$, and $T \in \mathcal{D}(X)$, then we have
    $$\pi^*\pi_*T=T+\sum_{p \in \emph{base}(\pi)} \langle [T] , [E_p]\rangle E_p$$
\end{lem}

\begin{proof}
    Siu's theorem implies that the difference $\pi^*\pi_*T-T$ is a current of integration on the exceptional divisor of $\pi$. Thus, we can write $\pi^*\pi_*T -T=\sum \lambda_j E_{p_j}$ where the $p_j$ are the points blown-up by $\pi$. We intersect this equality with the class $[E_{p_j}]$ and we obtain $\lambda_j=\langle [T],[E_{p_j}]\rangle$.
\end{proof}

\subsection{The Picard-Manin space}
We endow the space $\overline{C}(\mathbb{P}^2)$ with the Euclidean norm in the canonical basis, that is,  $$\Vert c \Vert^2:= a_L^2+\sum_{p \in \mathcal{B}(\mathbb{P}^2) } a_ p^2.$$
We define the \textbf{Picard-Manin space} to be the completion of  $(\overline{C}(\mathbb{P}^2), \Vert . \Vert )$ and we denote it by $l^2(\mathcal{B}(\mathbb{P}^2))$. More concretely, a Weil class $$c=a_L[L]+\sum_{p \in \mathcal{B}(\mathbb{P}^2) } a_p [E_p]$$
is in $l^2(\mathcal{B}(\mathbb{P}^2))$ if and only if $$\sum_{p \in \mathcal{B}(\mathbb{P}^2) } a_p^2 <+\infty.$$
The intersection form extends to the Picard-Manin space.

\subsection{Lelong numbers}
One key observation is that we are able to understand the singularities of a $(1,1)$ current on $\mathbb{P}^2$ by looking at the class of its strict transform in the Picard-Manin space. For any $p \in \mathcal{B}(\mathbb{P}^2)$ and $T \in W(\mathbb{P}^2)$, we define the Lelong number of $T$ at $p$ to be $$\nu(T,p):=\nu(T_{X,\pi},p)$$
whenever $p$ is defined on $(X,\pi)$. It is well-defined since a local isomorphism preserves the Lelong number. The following lemma is again a consequence of Siu's theorem \cite[Theorem 7.27]{GZ}. See \cite{Guedj} for a proof.

\begin{lem}\label{lem8}
    If $\phi: X \rightarrow Y$ is the blow-up of a point $p$ in a model $(Y,\pi)$ and $T_{Y}$ is a positive closed current of bidegree $(1,1)$ on $Y$, then we have
$$\phi^*T_{Y}=\widetilde{T_{Y}}+\nu(T_{Y},p)E_p.$$
\end{lem}

Applying Lemma \ref{lem8} to $T_{Y}=\widetilde{T_{\mathbb{P}^2}}$ yields the following.

\begin{lem}\label{lem1}
    Let $T_{\mathbb{P}^2}$ be a positive closed $(1,1)$ current on $\mathbb{P}^2$. Then, for every $p \in \mathcal{B}(\mathbb{P}^2)$, we have $$\nu(\widetilde{T_{\mathbb{P}^2}}, p)=\langle[\widetilde{T_{\mathbb{P}^2}}],[E_p]\rangle.$$
\end{lem}

The next lemma is due to Guedj \cite{Guedj} and it will be a key ingredient for the proof of Theorem \ref{thm1}. See also \cite[Theorem 7.1]{FJ} for a stronger version.

\begin{lem}[Guedj]\label{lem7}
    Let $T_{\mathbb{P}^2}$ be a positive closed $(1,1)$ current on $\mathbb{P}^2$ which does not charge curves. Then, we have $$\sum_{p \in \mathcal{B}(\mathbb{P}^2)} \nu(\widetilde{T_{\mathbb{P}^2}},p)^2 \leq [T_{\mathbb{P}^2}]^2.$$
\end{lem}

As a consequence, we obtain that the strict transform of a current on $\mathbb{P}^2$ that does not charge curves is in $l^2(\mathcal{B}(\mathbb{P}^2))$.

\subsection{Hyperbolicity}
The Picard-Manin space $l^2(\mathcal{B}(\mathbb{P}^2))$ is endowed with a intersection form which is non-degenerate and of signature $(1,\infty)$. We thus define  $$\mathbb{H}:=\{c \in l^2(\mathcal{B}(\mathbb{P}^2)) \mid c^2=1 \,  \text{and} \, \langle c,[L]\rangle >0\}. $$
We endow $\mathbb{H}$ with the metric
$$d_\mathbb{H}(c,c'):=\cosh^{-1}(\langle c, c'\rangle).$$
The space $\mathbb{H}$ is an infinite dimensional hyperbolic space. In particular, it is Gromov-hyperbolic. We can identify its Gromov boundary with the set of isotropic lines \cite[Section 3.5.1]{DSU},
$$\partial \mathbb{H}=\{c \in l^2(\mathcal{B}(\mathbb{P}^2)) \mid  c^2=0 \, \text{and} \, \langle c,L\rangle=1\}.$$
We note the following consequence of Cauchy-Schwarz inequality.
\begin{lem}\label{lem2}
Consider two classes $\theta, \theta' \in \partial \mathbb{H}$. We have $\langle \theta, \theta'\rangle \geq0$ with equality if and only if $\theta=\theta'$.
\end{lem}

\subsection{Functorial behaviour}
Let $f: X_1 \dashrightarrow X_2$ be a birational map between two models of $\mathbb{P}^2$. By Zariski's theorem \cite[Theorem 1.59]{La2}, there exist two models $(Y, \pi_1 \circ \alpha)$ and $(Y,\pi_2 \circ \beta)$ such that  $f\circ \alpha =\beta$. We call $\alpha$ a \textbf{resolution} of $f$.
\[
\begin{tikzcd}
    & Y \arrow[ld, "\alpha", swap] \arrow[rd, "\beta"] & \\
    X_1 \arrow[rr, dashed, "f"] & & X_2
\end{tikzcd}
\]
The \textbf{base points} of $f $ are the points blown-up in a minimal resolution of $f$.

Consider a birational map  $f \in \mathrm{Bir}(\mathbb{P}^2)$ and a Cartier current $T\in C(\mathbb{P}^2)$  defined on the model $(X,\pi)$. We denote by $f_{X,\pi}:=\pi^{-1}f\pi$ the induced birational map of $X$.
Consider a resolution $(Y,\pi \circ \alpha)$ of $f_{X,\pi}$,
\[
\begin{tikzcd}
    & Y \arrow[ld, "\alpha", swap] \arrow[rd, "\beta"] & \\
    X \arrow[rr, dashed, "f_{X,\pi}"] & & X.
\end{tikzcd}
\]
We define the pullback of $T$ by $$f^*T:=(\beta^*T_{X,\pi},(Y,\pi \circ \alpha)),$$
that is, $f^*T$ is the Cartier current defined on the model $(Y,\pi \circ \alpha)$ by the current $\beta^*T_{X,\pi}$. Note that the Cartier current $(\beta^*T_{X,\pi},(Y, \pi \circ \beta))$ is equal to $T$ by definition.
Now, if $T$ is a Weil current, we define the realization of $f^*T$ on $(X,\pi)$ by 
$$(f^*T)_{X,\pi}:=\alpha_*T_{Y,\pi \circ \beta}.$$
We can check that $f^*T$ is well-defined, and this definition coincides with the one given above for Cartier currents. Remark that the pullback preserves positivity.

\begin{prop}[Functoriality]
    We have $(f\circ g)^*=g^*f^*$ on $W(\mathbb{P}^2)$ for any $f,g \in \mathrm{Bir}(\mathbb{P}^2)$.
\end{prop}

\begin{proof}
See \cite[Lemma 6.30]{La2}.
\end{proof}

We can define similarly the pushforward by $f$, and we remark that we have $$f_*=(f^{-1})^*.$$
We denote by $f_{\mathbb{P}^2}:\mathcal{D}(\mathbb{P}^2) \rightarrow \mathcal{D}(\mathbb{P}^2)$ the usual pullback of currents on $\mathbb{P}^2$, which is defined by $$f_{\mathbb{P}^2}^*T_{\mathbb{P}^2}:=\alpha_*\beta^*T_{\mathbb{P}^2}$$
when $\alpha$ is a resolution of $f$ and $\beta=f \circ \alpha$.
We observe that the usual pullback $f_{\mathbb{P}^2}^*T_{\mathbb{P}^2}$ is equal to the realization $(f^*T_{\mathbb{P}^2})_{\mathbb{P}^2}$ on $\mathbb{P}^2$ of the pullback of the Cartier current $T_{\mathbb{P}^2}$.

We now consider the pullback at the level of classes $f^*:\overline{W}(\mathbb{P}^2) \rightarrow \overline{W}(\mathbb{P}^2)$. We gather a few properties of $f^*$ in the following proposition.

\begin{prop}\label{prop27}
    For any birational map $f$, the pullback $f^*$ defines an isometry of the Picard-Manin space $(l^2(\mathcal{B}(\mathbb{P}^2)), \langle , \rangle )$, which preserves the nef cone. Moreover, its adjoint is $f_*$, that is, we have $$\langle f^*c,d\rangle=\langle c, f_* d\rangle$$
    for any $c,d \in l^2(\mathcal{B}(\mathbb{P}^2)).$ We have also 
    $$\langle f^*[L],[L]\rangle= \deg(f).$$
\end{prop}

\begin{proof}
    See \cite[Corollary 1.113]{La2} and \cite[Lemma 6.32]{La2}.
\end{proof}

We obtain that $f^* $ defines an isometry of the infinite dimensional hyperbolic space $\mathbb{H}$. We denote by $\underline{f}: \partial \mathbb{H} \rightarrow \partial \mathbb{H}$ the induced action on the boundary, which is defined by
$$\underline{f}^*\theta:=\frac{1}{\langle f^*\theta, L \rangle} f^*\theta.$$
We can classify elements of $\mathrm{Bir}(\mathbb{P}^2)$ as elliptic, parabolic or loxodromic by their action on $\partial\mathbb{H}$. See \cite[Section 6.36]{La2}. An isometry is loxodromic if it has two fixed points on the boundary, one repelling and one attracting. We say that a subgroup (or a subsemigroup) of $\mathrm{Bir}(\mathbb{P}^2)$ is \textbf{non-elementary} if it contains two loxodromic elements with disjoint fixed sets.

\subsection{Sequential compactness}
Consider a countable subgroup $\Gamma$ of $\mathrm{Bir}(\mathbb{P}^2)$. We restrict the definition of $W(\mathbb{P}^2)$ to a countable set of models.
\begin{prop}\label{prop28}
    If $\Gamma$ is a countable subgroup of $\mathrm{Bir}(\mathbb{P}^2)$, then there exists a countable set of models $\mathcal{M}_\Gamma $ which is a projective system, and on which the pullback by elements of $\Gamma$ is well-defined. We define $$W_\Gamma(\mathbb{P}^2):=\lim_{\longleftarrow } \mathcal{D}(X,\pi) $$
    where the limit is taken over $(X,\pi) \in \mathcal{M}_\Gamma.$  We define similarly $C_{\Gamma}(\mathbb{P}^2)$ and the associated spaces of classes $\overline{W}_{\Gamma}(\mathbb{P}^2)$, $\overline{C}_\Gamma(\mathbb{P}^2)$.
\end{prop}

\begin{proof}
    We define $\mathcal{M}_0$ to be the set of all minimal resolutions of elements of $\Gamma$. It is a countable set of models, but it is not necessarily a projective system. We now define inductively an increasing sequence of countable sets of models. We first add to $\mathcal{M}_n$ all the minimal upper bounds of any two models in $\mathcal{M}_n$. Moreover, if $\gamma \in \Gamma$, we let $(Y, \alpha)$ be the minimal resolution of $\gamma$. For any $(X_1,\pi_1) \in \mathcal{M}_n$, we consider $(Z,\pi_1 \circ \phi)$ to be the minimal upper bound of $(X_1, \pi_1)$ and $(Y,\alpha)$. Then, we add the model $(Z,\beta \circ \psi)$ obtained in the following diagram

\[ \begin{tikzcd}
    & Z \arrow[ld, "\phi", swap] \arrow[rd, "\psi"] & &\\
    X_1 \arrow[rd, "\pi_1", swap]& & Y \arrow[ld, "\alpha" ] \arrow[rd, "\beta"] &\\
    &\mathbb{P}^2 \arrow[rr, "\gamma", dashed] & & \mathbb{P}^2.   
\end{tikzcd}
\]
This defines $\mathcal{M}_{n+1}$ from $\mathcal{M}_n$ and we define  
$$\mathcal{M}_\Gamma:=\bigcup_{n \in \mathbb{N}} \mathcal{M}_n$$
which is a countable set of models and a projective system.
We define $$W_\Gamma(\mathbb{P}^2)=\lim_{\longleftarrow} \mathcal{D}(X,\pi) $$
where $(X,\pi)$ is a model in $\mathcal{M}_\Gamma$.

We now check that the pullback by elements of $\Gamma$ is well-defined on $W_\Gamma(\mathbb{P}^2)$. Let $T \in W_\Gamma(\mathbb{P}^2) $ and  $(X,\pi)$ be a model in $\mathcal{M}_\Gamma$. We define $(\gamma^*T)_{X,\pi}$ using the following diagram:
\[
\begin{tikzcd}
    & &Z'' \arrow[ld, "u", swap] \arrow[rd, "v"] & & \\
    & Z \arrow[ld, "\phi", swap ] \arrow[rd, "\psi"]&  & Z' \arrow[ld, "\delta", swap] \arrow[rd, "\lambda"] & \\
    X \arrow[rd, "\pi", swap] & &  Y \arrow[rd, "\beta", swap] \arrow[ld,"\alpha"]& & X \arrow[ld, "\pi"] \\
    &\mathbb{P}^2 \arrow[rr, dashed, "\gamma"] &   & \mathbb{P}^2& \\
\end{tikzcd}
\]
The model $(Y,\alpha)$ is in $\mathcal{M}_\Gamma$ as a minimal resolution of $\gamma$. Remark that $(Y,\beta)$ is also in $\mathcal{M}_\Gamma$ as a minimal resolution of $\gamma^{-1}\in \Gamma$. Then, $(Z,\pi \circ \phi)\in \mathcal{M}_\Gamma$ is an upper bound of $(X,\pi)$ and $(Y,\alpha)$. Similarly, $(Z', \pi \circ \lambda)\in \mathcal{M}_\Gamma$ is an upper bound of $(Y,\beta) $ and $(X,\pi)$. Thus, $(Z, \beta \circ \psi)$ belongs to $\mathcal{M}_\Gamma$. Finally, $(Z'', \pi \circ \lambda \circ v)\in \mathcal{M}_\Gamma$ is an upper bound of $(Z,\beta \circ \psi)$ and $(Z', \pi \circ \lambda)$. Thus, $ \phi \circ u$ is a resolution of the induced map $\gamma_{X,\pi}=\pi^{-1}\gamma \pi$ and we can define $$(\gamma^*T)_{X,\pi}:=(\phi \circ u)_* T_{Z'', \pi \circ \lambda \circ v}.$$
We conclude that the pullback of any element of $W_\Gamma(\mathbb{P}^2)$ by $\gamma$ is well-defined.
\end{proof}

Consider a model $(X,\pi)$ in $\mathcal{M}_\Gamma$. We fix a Kähler form $\kappa_{X,\pi}$ on $X$. The mass of a current $T_{X,\pi} \in \mathcal{D}(X,\pi)$, is defined by $$\Vert T_{X,\pi} \Vert_{X,\pi} := \int_X T_{X,\pi} \wedge \kappa_{X,\pi}.$$
Note that we have $$\Vert T_{X,\pi} \Vert_{X,\pi}=\langle [T], [\kappa_{X,\pi}]\rangle$$
if $T_{X,\pi}$ is the realization on $(X,\pi)$ of a Weil current $T \in W_\Gamma(\mathbb{P}^2)$. Recall that a sequence of a positive closed $(1,1)$ currents on $(X,\pi)$ with bounded mass is relatively compact. We obtain the following result.  

\begin{lem}\label{lem3}
    If $(T_n)_{n \geq 1}\in W_{\Gamma}^+(\mathbb{P}^2)$ is a sequence of positive closed Weil currents such that $(\Vert T_{n 
    }\Vert_{X,\pi} )  $ 
    is bounded on each model $(X,\pi) \in \mathcal{M}_\Gamma$, then $(T_n)$ admits a convergent subsequence for the product topology.
\end{lem}

In the following, we will always denote by $W(\mathbb{P}^2)$ the space $W_{\Gamma}(\mathbb{P}^2)$ when we consider the dynamics of a finitely generated subgroup.

\section{Random walk on the Picard-Manin space}\label{section3}
In this section, we state results of Maher--Tiozzo \cite{MT2} and Gouëzel--Karlsson \cite{GK} concerning the asymptotic behaviour of the random walk on the Picard-Manin space. Consider a finitely supported probability measure $\mu$ on $\mathrm{Bir}(\mathbb{P}^2)$. We denote by $\Sigma_\mu$ its support and by $\Gamma_\mu$ the subsemigroup of $\mathrm{Bir}(\mathbb{P}^2)$ generated by $\Sigma_\mu$. We assume that $\Gamma_\mu$ is non-elementary. We set $\Omega:=\Sigma_\mu^\mathbb{N}$ and for $\omega=(f_n)_{n \geq 1} \in \Omega$, we consider the left product  $$f^n_\omega:=f_n \cdots f_1$$
for any $n \geq 1$.
We define also $f^0_\omega:=\mathrm{id}$.
The action by pullback induces a random walk on the Picard-Manin space. We choose $[L]$ as a base point for the random walk, that is, we consider $(f^n_\omega)^*[L]$. By functoriality, it defines a right walk on the infinite dimensional hyperbolic space $\mathbb{H}$:$$(f^n_\omega)^*[L]=(f_1)^* \cdots (f_n)^*[L].$$

\subsection{Random walks on Gromov hyperbolic spaces}
Random walks on proper Gromov hyperbolic spaces are well understood, and under reasonable assumptions on the law, we obtain classical properties such as transience, convergence to the boundary at linear speed and identification of the Poisson boundary with the Gromow boundary. Here, we cannot use these results as we are working with the infinite dimensional hyperbolic space $\mathbb{H}$ which is not proper. We state a theorem of Maher and Tiozzo which combines results from \cites{MT1, MT2} and that concerns random walks on non-proper Gromov hyperbolic spaces. 

Let $(X,d)$ be a Gromov hyperbolic space, and $\mu$ be a finitely supported probability measure on a subgroup $\Gamma $ of the isometry group of $X$. We denote by $\Sigma_\mu$ the support of $\mu$ and we set $\Omega:=\Sigma_\mu^\mathbb{N}$. For any $\omega=(f_n)_{n\geq 1} \in \Omega$, we define  $r^0_\omega :=id$ and $$ r^n_\omega:=f_1 \cdots f_{n}$$
when $n\geq 1$. Let $x_0 \in X$ be any base point. We consider the right random walk based on $x_0$, defined by $r^n_\omega(x_0)$.
Isometries of $X$ are classified as elliptic, parabolic or loxodromic \cite{DSU}. Loxodromic elements act with two fixed points on $\partial X$, one attracting, one repelling. We say that $\mu$ is non-elementary if the semigroup generated by $\Sigma_\mu$ contains two loxodromic elements with disjoint fixed point sets. We say that a loxodromic element $f \in \Gamma$ is weakly properly discontinuous (WPD) if for any $x \in X$ and $K\geq 0$, there exists $N\geq 1$, such that the set
$$\{ \gamma \in \Gamma \mid d(x, \gamma (x)) \leq K \, , \, d(f^N(x), \gamma f^N(x)) \leq K \}  $$
is finite. Intuitively, $f$ is WPD if its action is proper along its axis. We say that $\mu$ is WPD if the semigroup $\Gamma_\mu$ generated by $\Sigma_\mu$ contains a WPD element.

\begin{thm}[Maher--Tiozzo]\label{thm4}
    Let $\mu$ be a finitely supported, non-elementary probability measure on a subgroup $\Gamma$ of isometries of a Gromov hyperbolic space $(X,d)$. Fix a base point $x_0 \in X$.
    \begin{enumerate}
        \item For $\mu^\mathbb{N}$-almost every $\omega\in \Omega$, there exists a point $\xi (\omega) \in \partial X$ such that $$\lim_{n \rightarrow +\infty} r^n_\omega x_0 = \xi(\omega).$$
        \item There exists $l>0$ such that for $\mu^\mathbb{N}$-almost every $\omega\in \Omega$, $$\lim_{n \rightarrow +\infty} \frac{1}{n}d(r^n_\omega x_0 ,x_0) =l.$$
        \item We endow the Gromov boundary $\partial X$ with the hitting measure $\nu $ defined by 
        $$\nu(B):=\mu^\mathbb{N}(\omega  \mid \xi(\omega) \in B)$$
        for any Borel subset $B \subset \partial X$. The hitting measure is non-atomic, and if we assume moreover that $\mu$ is WPD, then $(\partial X, \nu)$ is a model for the Poisson boundary of the right random walk $(\Gamma,\mu)$, that is, the map $$\phi \mapsto \left( \gamma  \mapsto\int_{\partial X } \phi(\gamma (b)) \, d\nu(b)   \right) $$ 
        defines an isomorphism between $L^\infty(\partial X, \nu)$ and the space $H^\infty(\Gamma,\mu)$ of bounded $\mu$-harmonic functions.
    \end{enumerate}
\end{thm}

\begin{proof}
    The first part comes from \cite[Theorem 1.1]{MT1}, and the second from \cite[Theorem 1.2]{MT1}. 
    See \cite[Theorem 1.8]{MT2} for the third part.  
\end{proof}

\subsection{Degree growth}
We apply Theorem \ref{thm4} to the random walk on the Picard-Manin space.

\begin{thm}[Maher--Tiozzo]\label{thm2}
Let $\mu$ be a finitely supported, non-elementary probability measure on $\mathrm{Bir}(\mathbb{P}^2)$.
    \begin{enumerate}
        \item For $\mu^\mathbb{N}$-almost every $\omega \in \Omega$, there exists a nef class $\theta_\omega \in \partial\mathbb{H}$ such that 
        $$\lim_{n \rightarrow +\infty} \frac{1}{\deg(f^n_\omega)}(f^n_\omega)^*[L] = \theta_\omega$$
        in $l^2(\mathcal{B}(\mathbb{P}^2))$.
        \item There exists $A>0$ such that, for $\mu^\mathbb{N}$-almost every $\omega \in \Omega$, $$\lim_{n \rightarrow+\infty}\frac{1}{n}\log \deg(f^n_\omega) = A.$$
        \item Denote by $\nu$ the hitting measure on $\partial \mathbb{H}$, defined by $$\nu(A):=\mu^\mathbb{N}(\{ \omega \mid \theta_\omega \in A\}).$$
        The hitting measure is non-atomic, and if we assume moreover that $\mu$ is WPD, then $(\partial \mathbb{H}, \nu)$ is a model for the Poisson boundary of $(\mathrm{Bir}(\mathbb{P}^2),\mu)$.
    \end{enumerate}
\end{thm}

\begin{proof}
    The convergence of a sequence $(\theta_n) \in l^2(\mathcal{B}(\mathbb{P}^2))^\mathbb{N}$ normalized by $\langle \theta_n , [L]\rangle =1$ to a class $\theta $ in the Gromov boundary $\partial \mathbb{H}$ is equivalent to the convergence in $l^2$-norm.
    Thus, the first part follows from Theorem \ref{thm4}.(1) using that $$\langle (f^n_\omega)^*[L] , [L] \rangle= \deg(f^n_\omega).$$ Moreover, $\theta_\omega$ is nef as a limit of nef classes.
    The degrees are related to the drift of the walk by the following:
    \begin{align*}
        d_\mathbb{H}((f^n_\omega)^*[L],[L]) &= \cosh^{-1} (\langle f^n_\omega)^*[L], [L]\rangle) \\
        &=\cosh^{-1}(\deg(f^n_\omega))\\
        &\simeq \log \deg(f^n_\omega).
    \end{align*}
    and Theorem \ref{thm4}.(2) gives the second part.
    See also \cite[Theorem 1.2]{MT2}.  
\end{proof}

We write $$\theta_\omega=[L]-\sum_{p \in \mathcal{B}(\mathbb{P}^2)} \alpha_{p}(\omega) [E_p]$$
with $\alpha_p(\omega)\geq 0$ and $\sum \alpha_p(\omega)^2=1$.
Note that a loxodromic birational map is WPD if and only it is not conjugated to a monomial map \cite[Section 20.3]{La2}. The following lemma will be useful for the proof of Theorem \ref{thm1}.

\begin{lem}\label{lem4}
    Let $\mu$ be a finitely supported, non-elementary probability measure on $\mathrm{Bir}(\mathbb{P}^2)$. For $\mu^\mathbb{N}\otimes \nu$-almost every $(\omega,\theta) \in \Omega \times \partial \mathbb{H}$, there exists $C>0$ such that
    $$\langle f_n^*\cdots f_1^* \theta, [L] \rangle \geq C\deg(f_1\cdots f_n)$$
    for all $n \geq 1$.
\end{lem}

\begin{proof}
    Consider $(\omega, \theta)\in \Omega \times \partial\mathbb{H}.$ By Proposition \ref{prop27}, we have
    \begin{align*}
        \langle f_n^*\cdots f_1^* \theta, [L] \rangle &= \langle \theta, (f_{1})_* \cdots (f_{n})_* [L]\rangle \\
        &=\deg(f_1\cdots f_n) \left\langle \theta, \frac{1}{\deg(f_1 \cdots f_n)} (f_1 \cdots f_n)_*[L]\right\rangle.
    \end{align*}
    Define the reflected law $\check\mu$ by $\check\mu(g):=\mu(g^{-1})$ for any $g\in G$. Note that $\check\mu$ is non-elementary. We apply Theorem \ref{thm2} to the reflected law. We obtain the following: for $\check \mu^\mathbb{N}$-almost every $\omega \in \Omega$, there exists $\theta_\omega' \in \partial \mathbb{H}$ such that
    $$\frac{1}{\deg(f_1\cdots f_n)}(f_1\cdots f_n)_* [L] \longrightarrow \theta_\omega'.$$
    Now, by Theorem \ref{thm2}.(3), for any $\theta'_\omega \in \partial \mathbb{H}$, the subset $\{\theta=\theta_\omega'\}\subset \partial  \mathbb{H}$ has $\nu$-measure zero, and thus
    $$\mu^\mathbb{N}\otimes\nu (\{(\omega,\theta)\in \Omega \times \partial \mathbb{H} \mid \theta_\omega' = \theta \})=0.$$
    Then, by Lemma \ref{lem2}, we have $\langle \theta, \theta_\omega'\rangle >0$ for $\mu^\mathbb{N}\otimes\nu$-almost every $(\omega,\theta)$ and there exists $C=C_{\omega,\theta}>0$, such that $$\left\langle \theta, \frac{1}{\deg(f_1 \cdots f_n)} (f_1 \cdots f_n)_*[L]\right\rangle \geq C $$
    for $n$ large enough.
\end{proof}

\subsection{A result of Gouëzel and Karlsson}
The following theorem, due to Gouëzel and Karlsson \cite{GK} is a refinement of Kingman's subadditive theorem.

\begin{thm}[Gouëzel--Karlsson]\label{thm5}
We denote by $\sigma$ the shift on $\Omega$ and we consider a subadditive cocycle $a(n,.)$ over $(\Omega, \mu^\mathbb{N}, \sigma)$ with $a(1,.) \in L^1(\Omega, \mu^\mathbb{N})$. We define the asymptotic average $$A:=\inf \frac{1}{n } \int a(n,\omega) \, d\mu^\mathbb{N}(\omega).$$
For $\mu^\mathbb{N}$-almost every $\omega \in\Omega$, there exist an increasing sequence of integers $(n_i)$, and a sequence of positive real numbers $\delta_l \rightarrow 0$, such that, for every $i$ and every $l \leq n_i$, 
$$-l\delta_l \leq a(n_i,\omega)-a(n_i-l,\sigma^l\omega)-Al \leq l\delta_l.$$
\end{thm}

Applying Theorem \ref{thm5} to the subadditive cocycle $a(n,\omega):=\log \deg(f^n_\omega)$ yields the following: for $\mu^\mathbb{N}$-almost every $\omega\in\Omega$, there exist $n_i\rightarrow \infty$ and $\delta_l \rightarrow 0$ such that for every $i$ and every $l \leq n_i$, 
$$\frac{\deg(f^{n_i-l}_{\sigma^l(\omega)})}{\deg(f^{n_i}_\omega)} \leq \exp(-(A-\delta_l)l).$$
Remark that $A$ is equal to the drift of the walk given by Theorem \ref{thm2}.(2). The following lemma will be useful for the proof of Theorem \ref{thm1}.
\begin{lem}\label{lem5}
    Let $\mu$ be a finitely supported, non-elementary probability measure on $\mathrm{Bir}(\mathbb{P}^2)$. For $\mu^\mathbb{N}$-almost every $\omega \in \Omega$, there exist $C>0$ and $\rho\in (0,1)$ such that
    $$\langle (f^n_\omega)_*\theta_\omega, [L]\rangle \leq C \rho^n.$$
\end{lem}

\begin{proof}
    By Proposition \ref{prop27}, we have $ \langle (f^n_\omega)_*\theta_\omega, [L]\rangle = \langle \theta_\omega, (f^n_\omega)^*[L]\rangle $. We define $$\theta_n:=\frac{1}{\deg(f^n_\omega)}(f^n_\omega)^*[L]$$
    so that we have $ \langle (f^n_\omega)_*\theta_\omega, [L]\rangle =\deg(f^n_\omega)\langle \theta_\omega, \theta_n \rangle$.
    If $m >n$, then we can write
    \begin{align*}
        \langle \theta_m, \theta_n \rangle &= \frac{1}{\deg(f^m_\omega) \deg(f^n_\omega)} \langle (f^m_\omega)^*[L], (f^n_\omega)^*[L] \rangle \\
        &=\frac{1}{\deg(f^m_\omega) \deg(f^n_\omega)}\langle (f^{m-n}_{\sigma^n(\omega)})^*[L], [L] \rangle \\
        &=\frac{\deg(f^{m-n}_{\sigma^n(\omega)})}{\deg(f^m_\omega) \deg(f^n_\omega)}.
    \end{align*}
    Along the increasing sequence $(n_i) $ given by Theorem \ref{thm5}, we have
    $$\langle \theta_{n_i}, \theta_n \rangle  \leq \frac{1}{\deg(f^{n}_\omega)}  \exp(-(A-\delta_n)n).$$
    As $\theta_{n_i} $ converges to $ \theta_\omega$, we obtain $$\deg(f^n_\omega) \langle \theta_\omega, \theta_n \rangle \leq \exp(-(A-\delta_n)n)$$
    and this concludes the proof.   
\end{proof}

\section{Boundary currents}\label{section4}
Consider a finitely supported, non-elementary and WPD probability measure $\mu$ on $\mathrm{Bir}(\mathbb{P}^2)$. Let $S_{\mathbb{P}^2}\in \mathcal{D}(\mathbb{P}^2)$ be a positive closed $(1,1)$ current of mass one on $\mathbb{P}^2$. For any $\omega \in \Omega$, we consider the sequence of Cartier currents 
$$T_{n,\omega}:=\frac{1}{\deg(f^n_\omega)} (f^n_\omega)^*S_{\mathbb{P}^2}.$$
\begin{prop}\label{prop77}
    The sequence $(\Vert T_{n,\omega}\Vert_{X,\pi})$ is bounded for any model $(X,\pi)$.
\end{prop}
\begin{proof}
We have
    \begin{align*}
    \Vert T_{n,\omega} \Vert_{X,\pi} &= \frac{1}{\deg(f^n_\omega)}\langle (f^n_\omega)^*[S_{\mathbb{P}^2}] , [\kappa_{X,\pi}]\rangle \\
    & = \frac{1}{\deg(f^n_\omega)} \left\langle \deg(f^n_\omega)[L] - \sum a_p [E_p],\,  a_L [L] -\sum b_p [E_p] \right\rangle
\end{align*}
with $a_L,a_p, b_p \geq 0$ as the classes $(f^n_\omega)^*[L]$ and $[\kappa_{X,\pi}]$ are nef. Thus, we obtain 
$$\Vert T_{n,\omega} \Vert_{X,\pi} \leq a_L$$
and this concludes the proof.
\end{proof}

We denote by $\mathcal{E}_\omega$ the set of all limits of subsequences of $(T_{n,\omega})_{n \geq 1}$. We denote also by $\mathcal{C}_\omega$ the set of positive closed Weil currents that are cohomologous to $\theta_\omega$. By Theorem \ref{thm2}, we have $\mathcal{E_\omega} \subset \mathcal{C}_\omega$. We work with a countable set of models given by Proposition \ref{prop28} and we obtain that $\mathcal{E}_\omega$ is non empty for $\mu^\mathbb{N}$-almost every $\omega \in \Omega$ by Lemma \ref{lem3} and Proposition \ref{prop77}. Thus, to show that $(T_{n,\omega})_{n \geq 1}$ converges, it suffices to show that $\mathcal{E}_\omega$ contains at most one current.

It appears that it is easier to work with limit currents that are diffuse, that is, they do not charge curves in any model. In particular, such currents are strict transforms of their realization on $\mathbb{P}^2$. Denote by $\mathcal{C}^0_\omega \subset \mathcal{C}_\omega$ the subset of currents that are cohomologous to $\theta_\omega$ and that are diffuse. A crucial observation is the following: by Lemma \ref{lem1}, two currents in $\mathcal{C}^0_\omega$ have the same Lelong numbers at every point, in every model. We may hope that $\mathcal{E}_\omega \subset \mathcal{C}^0_\omega$, that is, limit currents do not charge curves, but we are not able to prove it in general. In this section, we show that this is the case under the following hypothesis: all the base points of $f^n_\omega$ are almost surely on $\mathbb{P}^2$ for all $n\geq 1.$ We will show in the next section that this assumption is satisfied for a generic $\mu$.

\subsection{Base points}
Recall that the base points of a birational map are the points blown-up in a minimal resolution. These points are in $\mathcal{B}(\mathbb{P}^2)$ and we have the following dichotomy: a base point $p$ is either defined on $\mathbb{P}^2$ and we say that $p$ is \textbf{proper}, or $p$ has order at least one. Throughout this section, we work under the following assumption.

\begin{hyp}\label{hyp1}
    We suppose that $\mu$ satisfies the following: for $\mu^\mathbb{N}$-almost every $\omega=(f_n)_{n\geq 1}$ and every $n\geq 1$, all the base points of $f^n_\omega$ are proper.
\end{hyp}

In particular, we can write $$\theta_\omega=[L]-\sum_{p \in \mathbb{P}^2} \alpha_p(\omega) [E_p]$$
with $\alpha_p(\omega) \geq 0$ and $\sum \alpha_p(\omega)^2 =1.$

\subsection{Limit currents do not charge curves on the plane}

The above assumption allows us to show that the limit currents in $\mathcal{E}_\omega$ do not charge curves in any model. We first prove it in the case of $\mathbb{P}^2$. Actually, we prove the stronger result that any current in $\mathcal{C}_\omega$ do not charge curves in any model.

\begin{prop}\label{prop1}
    We work under Assumption \ref{hyp1}. For $\mu^\mathbb{N}$-almost every $\omega \in \Omega$, for every $T\in \mathcal{C}_\omega$, the realization of $T$ on $\mathbb{P}^2$ does not charge any curve. 
\end{prop}

\begin{proof}
    Consider $T\in \mathcal{C}_\omega$. Define $\mathcal{B}_\omega$ to be the union of base points of $f^n_\omega$ for $n\geq 1$. By Assumption \ref{hyp1}, $\mathcal{B}_\omega$ is almost surely a countable subset of $\mathbb{P}^2$. We will show that the Lelong number of $T_{\mathbb{P}^2}$ is equal to zero for any $q \in \mathbb{P}^2 \setminus \mathcal{B}_\omega$. We denote by $S_n:=(f^n_\omega)_*T$ the pushforward of $T$ by $f^n_\omega$, and by $m_n:=\Vert S_{n,\mathbb{P}^2}\Vert $ the mass of its realization on $\mathbb{P}^2$. We then consider the normalized sequence 
    $$R_n:=\frac{1}{m_n}S_n$$
    so that we have $T=m_n (f^n_\omega)^*R_n $ and the sequence $(R_{n,\mathbb{P}^2})$ is relatively compact. We have $$\nu(T,q)=m_n \nu((f^n_\omega)^*R_n,q).$$
    The next two lemmas show that $m_n\rightarrow 0$ and $\nu((f^n_\omega)^*R_n,q)$ is bounded for any $q \in \mathbb{P}^2 \setminus \mathcal{B}_\omega$.
    
    \begin{lem}\label{lem31}
        The sequence $(m_n)$ converges to zero for $\mu^\mathbb{N}$-almost every $\omega \in \Omega$.
    \end{lem}

    \begin{proof}
        Recall that we have
       $ m_n=\langle [S_{n,\mathbb{P}^2}],[L]\rangle.$
    Consider a resolution $(X_n,\pi_n)$ of $f^n_\omega$ and define $\sigma_n:=f^n_\omega \circ \pi_n$ so that we have $S_{n, \mathbb{P}^2}=(\sigma_{n})_ * T_{X_n,\pi_n}.$
    Using that $T$ is cohomologous to $\theta_\omega$, we obtain 
    \begin{align*}
        m_n&=\langle (\sigma_n)_*\theta_{\omega, X_n,\pi_n} , [L] \rangle_{\mathbb{P}^2} \\
        &=\langle \theta_{\omega, X_n,\pi_n} , \sigma_n^*[L] \rangle_{X_n,\pi_n}\\
        &=\langle \theta_{\omega, X_n,\pi_n} , (f^n_\omega)^*[L] \rangle_{X_n,\pi_n}
    \end{align*}
    As $(f^n_\omega)^*L$ is a Cartier class defined on $(X_n,\pi_n)$, we have $m_n=\langle \theta_\omega, (f^n_\omega)^*[L] \rangle$ and we conclude by Lemma \ref{lem5}. 
    \end{proof}

    \begin{lem}\label{lem30}
        The sequence $\nu((f^n_\omega)^*R_n,q)$ is bounded for $\mu^\mathbb{N}$-almost every $\omega \in \Omega$ and for every $q\in \mathbb{P}^2 \setminus \mathcal{B}_\omega$.
    \end{lem}

    \begin{proof}
    The realization on $\mathbb{P}^2$ of $(f^n_\omega)^*R_n$ is equal to $(\pi_n)_*R_{n, X_n, \sigma_n}$. As $q$ is not a base point of $f^n_\omega$, the map $\pi_n$ defines a local isomorphism at $q$ and we have
    $$\nu((f^n_\omega)^*R_n, q)=\nu(R_{n, X_n, \sigma_n}, \pi_n^{-1}(q)).$$
    We have $\sigma_n^*R_{n,\mathbb{P}^2}=\sigma_n^*(\sigma_n)_* R_{n,X_n,\sigma_n}$, and thus, by Lemma \ref{lem6}, we have 
    $$\sigma_n^*R_{n,\mathbb{P}^2} = R_{n,X_n,\sigma_n}+\sum_{p \in \text{base}(\sigma_n)} \langle [R_{n, X_n, \sigma_n}], [E_p]\rangle E_p. $$
    As $R_n$ is nef, we obtain $\nu((f^n_\omega)^*R_n, q) \leq \nu(\sigma_n^*R_{n,\mathbb{P}^2}, \pi_n^{-1}(q)).$
    Moreover, $(R_{n,\mathbb{P}^2})$ is a relatively compact sequence of currents, and each $\sigma_n$ consists on blow-ups of points in $\mathbb{P}^2$ so we have
    $$\sup_{n} \nu(\sigma_n^*R_{n,\mathbb{P}^2}, \pi_n^{-1}(q)) < \infty$$
    and this concludes the proof.
    \end{proof}
    By Lemmas \ref{lem31} and \ref{lem30} we obtain that the Lelong number of $T_{\mathbb{P}^2}$ at any point $q$ which is not in the countable set $\mathcal{B}_\omega$ is zero. Thus, $T_{\mathbb{P}^2}$ does not charge curves.
\end{proof}

\subsection{Limit currents are strict transforms}

We now upgrade Proposition \ref{prop1} to any model.

\begin{prop}\label{prop2}
We work under Assumption \ref{hyp1}. For $\mu^\mathbb{N}$-almost every $\omega \in \Omega$ and every current $T\in \mathcal{C}_\omega$, the realization of $T$ on any model does not charge curves.
\end{prop}

\begin{proof}
    Consider first a model $(X,\pi)$ which is the blow-up of a single point $p \in \mathbb{P}^2$. By Lemma \ref{lem6}, we have $$\pi^*\pi_*T_{X,\pi}=T_{X,\pi}+\alpha_p(\omega) E_p.$$
    Moreover, we have $\pi_*T_{X,\pi}=T_{\mathbb{P}^2}$, and Lemma \ref{lem8} implies that $$\pi^*T_{\mathbb{P}^2}=\widetilde{T_{\mathbb{P}^2}}+\nu(T_{\mathbb
    {P}^2}, p) E_p.$$
    We obtain
    $$T_{X,\pi}=\widetilde{T_{\mathbb{P}^2}}+(\nu(T_{\mathbb{P}^2},p)-\alpha_p(\omega))E_p.$$
    As $T_{X,\pi}$ is a positive current, Siu's theorem implies that
    $$\alpha_p(\omega) \leq \nu(T_{\mathbb{P}^2},p). $$
    Thus, we have 
    $$1= \sum_{p \in \mathbb{P}^2}\alpha_p(\omega)^2\leq \sum_{p \in \mathbb{P}^2} \nu(T_{\mathbb{P}^2}, p)^2 \leq \sum_{p \in \mathcal{B}(\mathbb{P}^2)} \nu(\widetilde{T_{\mathbb{P}^2}}, p)^2  \leq 1$$
    where the first equality comes from the fact that $\theta_\omega^2=1$, and the last inequality is a consequence of Lemma \ref{lem7}. We deduce that $\nu(T_{\mathbb{P}^2},p) = \alpha_p(\omega)$ if $p \in \mathbb{P}^2$ and $\nu(\widetilde{T_{\mathbb{P}^2}}, q)=0$ if $q \notin \mathbb{P}^2$. We conclude that $T_{X_1,\pi_1}=\widetilde{T_{\mathbb{P}^2}}$ when $(X_1,\pi_1)$ is the blow-up of a single point $p_1 \in \mathbb{P}^2$. 
    
    If $(X_2,\pi_2)$ is the blow-up of a point $p_2  \in X_1$, we have
    \begin{align*}
        T_{X_2}&= \widetilde{T_{X_1}} + (\nu(T_{X_1}, p_2) - \alpha_{p_2}(\omega))E_{p_2}\\
        &=\widetilde{T_{\mathbb{P}^2}}+(\nu(\widetilde{T_{\mathbb{P}^2}}, p_2) - \alpha_{p_2}(\omega))E_{p_2}\\
        &=\widetilde{T_{\mathbb{P}^2}}.
    \end{align*}
    We obtain the result by induction on the number of blow-ups.
\end{proof}

\section{Groups generated by generic birational maps}\label{section 5}

In this section, we describe the structure of the group generated by generic birational maps. For any $d\geq 1$, we denote by $\mathrm{Bir}_d(\mathbb{P}^2)$ the set of birational maps of degree $d$.  Recall that $\mathrm{Bir}_d(\mathbb{P}^2)$ is a quasiprojective variety \cite[Chapter 5]{La2}.

\begin{thm}\label{thm3}
    For any $r\geq 2$ and $d_1, \dots ,d_r \in \mathbb{N}^\times$ with $\min d_i \geq 2$, there exists a Zariski dense subset $V\subset \mathrm{Bir}_{d_1}(\mathbb{P}^2) \times \cdots \times \mathrm{Bir}_{d_r}(\mathbb{P}^2)$ such that for any $(g_1,\dots, g_r)\in V$, the following hold.
    \begin{enumerate}
        \item The group $\Gamma:=\langle g_1,\dots, g_r \rangle
        $ is a free group of rank $r$.
        \item If $\gamma=\gamma_1 \cdots \gamma_n \in \Gamma$ is a reduced word, that is, $\gamma_i \in \{g_1^{\pm 1}, \dots, g_r^{\pm1}\}$ and $\gamma_i^{-1}\neq \gamma_{i-1}$ for every $i \leq n$, then $$\deg(\gamma)=\prod_{i=1}^n \deg(\gamma_i) .$$
        \item Any element of $\Gamma$ has only proper base points.
    \end{enumerate}
\end{thm}

\begin{proof}
    Consider a reduced word $w(g_1,\dots,g_r)=\gamma_1\cdots \gamma_n$ with $\gamma_i \in \{g_1^{\pm 1}, \dots, g_r^{\pm1}\}$ and $\gamma_i^{-1}\neq \gamma_{i-1}$ for every $i\leq n$. We write $w(g_1,\dots, g_r)=[h_0:h_1:h_2]$ where $h_i \in \mathbb{C}[x,y,z]$ are homogenous polynomials of the same degree, with coefficients that depend algebraically on $(g_1, \dots,g_r)$. Remark that the degree of $w(g_1, \dots, g_r)$ is strictly smaller than the product of the degrees of the $\gamma_i$ if and only if $h_0,h_1,h_2$ have a common factor. This condition is given by the vanishing of some resultant and thus describes a proper subvariety of $\mathrm{Bir}_{d_1}(\mathbb{P}^2) \times \cdots \times \mathrm{Bir}_{d_r}(\mathbb{P}^2)$.

    Recall that the Jacobian of a birational map $f=[f_0:f_1:f_2]$ is defined by $\mathrm{Jac}(f):=\det(\partial f_i/\partial x_j)$ and is an homogenous polynomial of degree $3\deg(f)-3$. We use the following criterion \cite[Proposition 2.40]{La2}: a birational map $f \in \mathrm{Bir}(\mathbb{P}^2)$ has only proper base points if and only if the Jacobian $\mathrm{Jac}(f^{-1})$ is not reduced. Thus, the fact that $w(g_1,\dots, g_r)$ has a non-proper base point is also given by the vanishing of some resultant and defines a proper subvariety of $\mathrm{Bir}_{d_1}(\mathbb{P}^2) \times \cdots \times \mathrm{Bir}_{d_r}(\mathbb{P}^2)$.

    We set $V$ to be the complement in $\mathrm{Bir}_{d_1}(\mathbb{P}^2) \times \cdots \times \mathrm{Bir}_{d_r}(\mathbb{P}^2)$ of the countable union of proper subvarieties defined above. We remark that $V$ is Zariski dense, that is, $V$ intersects any irreducible component of $\mathrm{Bir}_{d_1}(\mathbb{P}^2) \times \cdots \times \mathrm{Bir}_{d_r}(\mathbb{P}^2)$. Indeed, fix an  irreducible component $\mathcal{C}_i \subset \mathrm{Bir}_{d_i}(\mathbb{P}^2)$ for any $i \leq r $. By \cite[Lemma 8.22]{La2}, there exists $(g_1, \dots,g_r)\in \mathcal{C}_{i_1} \times \cdots \times \mathcal{C}_{i_r}$ such that each $g_i^{\pm 1}$ has only proper base points. Then, for generic $\sigma_1, \dots, \sigma_r \in \mathrm{PGL}_3$, the map $w(\sigma_1g_1\sigma_1^{-1}, \dots, \sigma_r g_r  \sigma_r^{-1})$ have only proper base points. 
\end{proof}

Theorem \ref{thm3} has the following consequence: if $\mu$ is a measure supported on $\{g_1^{\pm 1}, \dots, g_r^{\pm 1}\}$ with $(g_1,\dots, g_r)\in V$, then Assumption \ref{hyp1} is satisfied.

\section{Convergence of pullbacks of currents}\label{section6}
\subsection{Setup}\label{setup}
We fix $r\geq 2$ and $d_1, \dots , d_r \in \mathbb{N}^\times$ with $\min d_i \geq 2$. We consider a measure $\mu$ supported on  $\Sigma_\mu=\{g_1^{\pm1}, \dots, g_r^{\pm 1}\}$
with $(g_1, \dots g_r) \in V$ given by Theorem \ref{thm3}.

\begin{prop}
    The (semi)group $\Gamma_\mu$ generated by $\Sigma_\mu$ is non-elementary and contains a WPD element.
\end{prop}

\begin{proof}
    By Theorem \ref{thm3}, we have $\deg(g_1^n)=d_1^n$, thus $g_1$ is a loxodromic element with dynamical degree equal to $d_1$. By \cite[Corollary 20.24]{La2}, $g_1$ is a WPD element of $\Gamma_\mu$. The group $\Gamma_\mu$ is a free group of rank $r\geq 2$ that contains only loxodromic elements, and thus is non-elementary by \cite[Lemma 7.3]{Urech}.
\end{proof}

Let $S_{\mathbb{P}^2}\in \mathcal{D}(\mathbb{P}^2)$ be a positive closed $(1,1)$ current of mass one on $\mathbb{P}^2$. Recall that we define $\mathcal{C}_\omega$ to be the set of positive closed Weil currents cohomologous to $\theta_\omega$ and $\mathcal{C}^0_\omega \subset \mathcal{C}_\omega$ to be the subset of diffuse currents. We consider the sequence of pullbacks $$T_{n,\omega}:=\frac{1}{\deg(f^n_\omega)}(f^n_\omega)^*S_{\mathbb{P}^2}$$
and we define $\mathcal{E}_\omega$ to be the set of all limits of subsequences of $T_{n,\omega}$. By Proposition \ref{prop2} and Theorem \ref{thm3}, we have that $\mathcal{E}_\omega$ is a subset of $\mathcal{C}^0_\omega$. 
In particular, currents of $\mathcal{E}_\omega$ are strict transforms of their realization on $\mathbb{P}^2$, and we can identify $\mathcal{E}_\omega$ with a subset of $\mathcal{D}(\mathbb{P}^2)$. For any $\theta \in \partial \mathbb{H}$, we define $\mathcal{C}(\theta)$ to be the set of positive closed $(1,1)$ currents $T_{\mathbb{P}^2}$ on $\mathbb{P}^2$, satisfying \begin{enumerate}
    \item $T_{\mathbb{P}^2}$ does not charge curves
    \item  $[\widetilde{T_{\mathbb{P}^2}}]=\theta$
\end{enumerate}
so that we can identify $\mathcal{C}^0_\omega$ to $\mathcal{C}(\theta_\omega)$.

\begin{prop}\label{prop3}
For $\nu$-almost every $\theta \in \partial\mathbb{H}$, the set $\mathcal{C}(\theta)$ contains at most one current.
\end{prop}

Lemma \ref{lem3} and Proposition \ref{prop3} gives that $\mathcal{E}_\omega=\{T_\omega\}$ is almost surely a singleton. This implies that $(T_{n,\omega})$ converges to $T_\omega$ and we obtain convergence on $\mathbb{P}^2$ as a special case: $$\frac{1}{\deg(f^n_\omega)}(f^n_\omega)^*_{\mathbb{P}^2}S_{\mathbb{P}^2} \longrightarrow T_{\omega,\mathbb{P}^2}.$$
This proves Theorem \ref{thm1}.(1).

We have also that $[T_\omega] \neq [T_{\omega'}]$
for $\mu^\mathbb{N}\otimes \mu^\mathbb{N} $-almost every $(\omega, \omega')\in \Omega \times \Omega$ by Theorem \ref{thm2}.(3). This proves Theorem \ref{thm1}.(2). Furthermore, we can identify the Poisson boundary of the walk to a space of currents.

The rest of the paper is devoted to the proof of Proposition \ref{prop3}. We combine volume estimates for the random dynamical system on $\mathbb{P}^2$ with Skoda's uniform integrability theorem \cite[Theorem 8.11]{GZ} to show that $\mathcal{C}(\theta)$ contains at most one current.

\subsection{Volume estimates}

The next proposition says that volumes do not decrease too fast under the action of $\Gamma$. The volumes are computed here with the Fubini-Study metric on $\mathbb{P}^2$.

\begin{prop}\label{prop5}
Let $(g_1,\dots, g_r) \in V$. There exist $C_1, C_2>0$ such that $$\mathrm{vol}(\gamma(B))\geq (C_1\mathrm{vol}B)^{C_2 \deg(\gamma)}$$
for every $\gamma \in \Gamma=\langle g_1,\dots, g_r\rangle$ and every Borel subset $B \subset \mathbb{P}^2$.
\end{prop}

\begin{proof}
    We only give a sketch of the proof as it is an adaptation of well-known arguments in holomorphic dynamics. We refer to \cite[Proposition 1.2]{Guedj2} and \cite[Theorem 3.1]{Bay} for the details. We denote by $\mathrm{Jac}(\gamma)$ the (real) Jacobian of $\gamma$ for the Fubini-Study form, that is defined by the equality $\gamma^*\omega_{FS}^2=\mathrm{Jac}(\gamma) \omega_{FS}^2$. A computation in local charts shows that the function $\log |\mathrm{Jac}(\gamma)|$ can be written as a difference of functions that are $6\deg(\gamma) \omega_{FS}$-psh. We write $$\log |\mathrm{Jac}(\gamma)|=u_\gamma^+ - u_\gamma^- +c_\gamma$$
    where $c_\gamma$ is a constant and $u_\gamma^\pm$ are $6\deg(\gamma)\omega_{FS}$-psh functions, normalized by $\sup u_\gamma^\pm=0$. 
    
    \begin{lem}
        There exists $C>0$ such that $|c_\gamma|\leq C \deg(\gamma)$ for every $\gamma \in \Gamma$.
    \end{lem}

    \begin{proof}
    It is sufficient to show that the functions $\frac{1}{\deg(\gamma)}\log |\mathrm{Jac}(\gamma) |$ are uniformly bounded in $L^1(\mathbb{P}^2)$. We write $\gamma=\gamma_n \cdots \gamma_1$ with $\gamma_i \in \{g_1^{\pm 1}, \dots , g_r^{\pm 1}\}$ and $\gamma_{i+1}\neq \gamma_i^{-1}$. Thus, we have $$\log |\mathrm{Jac}(\gamma)|=\sum_{i=1}^n \log|\mathrm{Jac}(\gamma_i)| \circ \gamma_{i-1} \dots \gamma_1$$
    We can write $\log|\mathrm{Jac}(\gamma_i)|=u_i-v_i$ where $u_i$ and $v_i $ are in a finite family of $C\omega_{FS}$-psh functions for some $C>0$. Then, a straightforward generalization of \cite[Proposition 1.3]{Guedj2} gives 
    $$\Vert \log| \mathrm{Jac} (\gamma) | \Vert_{L^1(\mathbb{P}^2)} \leq C'\sum_{i=0}^n \deg(\gamma_i \cdots \gamma_1) \leq C'' \deg(\gamma)$$
    and this concludes the proof of the lemma.
    \end{proof}
    The desired volume estimates then follows from Skoda's uniform integrability theorem \cite[Theorem 8.11]{GZ}.
\end{proof}

We insist on the fact that these estimates rely heavily on the good control on the degree of elements of $\Gamma$, that is, we have $$\deg(\gamma)=\deg(\gamma_n) \cdots  \deg(\gamma_1)$$ whenever the writing $\gamma=\gamma_n \cdots \gamma_1$ is reduced. In general, we are not able to prove sufficient volume estimates for random products of birational maps.

\subsection{Potentials}
Consider a current $T\in \mathcal{C}(\theta)$. We denote by $\kappa$ the Fubini-Study form on $\mathbb{P}^2$ and we write $T=\kappa +dd^c u_T$
with $u_T$ a $\kappa$-psh function, normalized by $$\int_{\mathbb{P}^2 } u_T \, \kappa^2 = 0.$$
Following Cantat--Dujardin \cite[Section 6]{CD1}, we define the \textbf{diameter} of a boundary class  $\theta \in \partial \mathbb{H}$ by
$$\text{diam}(\theta):= \sup_{T,S \in \mathcal{C}(\theta)} \int_{\mathbb{P}^2} |u_T-u_S|.$$
This function is well defined on $(\partial \mathbb{H}, \nu)$ since $\mathcal{C}(\theta)$ is almost surely non-empty and $\text{diam}$ is a bounded function since the space of normalized $\kappa$-psh functions is relatively compact for the $L^1$ topology.

\begin{lem}\label{lem14}
    The diameter function is upper semi-continuous and thus defines a mesurable function on $(\partial \mathbb{H}, \nu)$. 
\end{lem}

\begin{proof}
    Let $(\theta_n)$ be a sequence of $(\partial \mathbf{H},\nu)$ that converges to $\theta$. Let $(T_n)$ and $(S_n)$ be sequences of currents in $\mathcal{C}(\theta_n)$ such that $$\int |u_{T_n}-u_{S_n}| \geq \text{diam}(\theta_n) - \frac{1}{n}.$$
    Up to pass to a subsequence, we can assume $$\text{diam}(\theta_{n}) \rightarrow \limsup_n \text{diam}(\theta_n).$$
    Moreover, $(\widetilde{T_n})$ and $(\widetilde{S_n})$ have bounded mass in every model. Up to extract, $\widetilde{T_n}$ converges to $T \in W^+(\mathbb{P}^2)$ and $\widetilde{S_n}$ converges to $S\in W^+(\mathbb{P}^2)$. Remark that $T,S$ are cohomologuous to $\theta$ and by Proposition \ref{prop2}, $T$ and $ S$ belongs to $\mathcal{C}(\theta)$. By dominated convergence, we obtain
    $$\int |u_T-u_S|=\limsup \text{diam}(\theta_n)$$
    and thus $\text{diam}(\theta) \geq \limsup \text{diam}(\theta_n).$
\end{proof}

We now study the action of a birational map by pullback on potentials.

\begin{prop}\label{prop6}
    If $S$ and $S'$ are two currents in $\mathcal{C}(f^*\theta)$ with normalized potentials $u_S$ and $ u_{S'}$, then there exist $T$ and $T'$ in $\mathcal{C}(\theta)$ such that 
    $$u_S-u_{S'}=u_T\circ f - u_{T'} \circ f +C $$
    where $u_T$ and $u_{T'} $ are normalized potentials of $T$ and $T'$, and 
    $$C=\int u_{T'} \circ f - u_T\circ f.$$
\end{prop}

\begin{proof}
     We write $S=\lambda \kappa +dd^c u_S$ and $S'=\lambda \kappa +dd^c u_{S'}$ with $\lambda=\langle f^* \theta, [L] \rangle$. We define $T:=f_*\widetilde{S}$ and $T':=f_*\widetilde{S'}$. These two currents are cohomologous to $\theta$. We show they are in $\mathcal{C}(\theta)$, that is, they do not charge curves in any model.
     \begin{lem}\label{lem15}
         If $\widetilde{R} \in W(\mathbb{P}^2)$ is the strict transform of a current $R \in \mathcal{D}(\mathbb{P}^2)$ that does not charge curves on $\mathbb{P}^2$, then for any $f \in \mathrm{Bir}(\mathbb{P}^2)$, the Weil current $f^*\widetilde{R}$ does not charge curves in any model.
     \end{lem}

     \begin{proof}
        Consider a model $(X,\pi)$ and a resolution $(Z,\pi \circ \alpha)$ of $f_{X,\pi}=\pi^{-1} f \pi$ so that we have the following diagram 
         \[ 
         \begin{tikzcd}
            & Z \arrow[ld, "\alpha", swap] \arrow[rd, "\beta"] & \\
            X\arrow[rr, dashed, "f_{X,\pi}"]& & X  
        \end{tikzcd} 
        \]
        The realization of $f^*\widetilde{R}$ on $(X,\pi)$ is given by $\alpha_*\widetilde{R}_{Z,\pi \circ \beta}$ and thus do not charge curves.
     \end{proof}

     Consider a resolution $(X,\alpha)$ of $f$ and $\beta:=f\circ \alpha$ so that we have $T=\beta_*\widetilde{S}_{X,\alpha}$. Then, we consider the usual pullback
        $$ (f^*T)_{\mathbb{P}^2}=\alpha_* \beta^* \beta_* \widetilde{S}_{X,\alpha} $$
    By Lemma \ref{lem6}, we can write         
         $(f^*T)_{\mathbb{P}^2}= S+E $
    with $$E:=\sum_{p \in \text{base}(\beta)} \langle [S_{X,\alpha}], [E_p] \rangle \alpha_*E_p.$$
    Similarly, we have $(f^*T')_{\mathbb{P}^2} =S'+E$
    with the same exceptional current $E$ since $\widetilde{S}$ and $\widetilde{S'}$ are cohomologous.
    We obtain $S-S'=f^*_{\mathbb{P}^2}(T-T')$
    and thus the potentials $u_S-u_{S'}$ and $u_T\circ f -u_{T'}  \circ f$ coincide up to a constant.
\end{proof}

As a corollary, we obtain the following result which allows us to use a dynamical argument to show that the diameter function is zero on $(\partial \mathbb{H}, \nu)$.

\begin{prop}\label{prop7}
    We have $$\emph{diam}(\underline{f}^*\theta)\leq \frac{2}{\langle f^*\theta , [L] \rangle } \sup_{T,T' \in \mathcal{C}(\theta)} \int |u_T \circ f - u_{T'} \circ f |.$$    
\end{prop}

\subsection{Skoda's Theorem}

\begin{lem}\label{lem16}
    For every $\alpha>0$ and for $\nu$-almost every $\theta \in \partial \mathbb{H}$, there exists $C>0$, such that for any $T, T' \in  \mathcal{C}(\theta)$, with normalized potentials $u_T, u_{T'}$, then 
    $$\emph{vol}(|u_T-u_{T'}| \geq t) \leq C\exp(-\alpha t)$$
    for every $t >0$.
\end{lem}

\begin{proof}
    It suffices to show that the functions $\exp(\alpha |u_T-u_{T'}|)$ are uniformly bounded in $L^1(\mathbb{P}^2)$ when $T,T' \in \mathcal{C}(\theta)$.  Consider a model $(X,\pi)$ that will be chosen conveniently later. We endow $X$ with a Kähler form $\kappa_{X,\pi}$ and we consider $C>0$ such that $\pi^*\text{vol}_{\mathbb{P}^2} \leq C \text{vol}_{X,\pi}.$ Thus, we have 
    $$\int_{\mathbb{P}^2}  \exp(\alpha |u_T-u_{T'}|) \leq C \int_X \exp(\alpha |u_T \circ \pi -u_{T'} \circ \pi|) $$
    We will find quasi-psh functions $v_T$ and $v_{T'}$ with small Lelong numbers satisfying  $u_T\circ \pi - u_{T'} \circ \pi=v_T -v_{T'}$. By Lemmas \ref{lem1} and \ref{lem7}, for any $\epsilon>0$, there exists a model $(X,\pi)$ on which the Lelong numbers of $\widetilde{T}$ and $\widetilde{T'} $ are smaller than $\epsilon$. Indeed, it suffices to consider the model obtain by blowing-up all the points satisfying $\langle \theta, [E_p]\rangle \geq \epsilon$.
    Moreover, by Lemmas \ref{lem8} and \ref{lem1}, we have $\pi^*T-\pi^*T'=\widetilde{T}- \widetilde{T'}.$
    Thus, $$dd^c (u_T\circ \pi - u_{T'} \circ \pi )=dd^c(v_T -v_{T'})$$
    where $v_{T}, v_{T'}$ have Lelong numbers smaller than $\epsilon$.
    We normalize by 
    $$\int_{\mathbb{P}^2} v_T \circ \pi^{-1}=0$$
    and similarly for $v_{T'}$ so that $u_{T}\circ \pi -u_{T'} \circ \pi=v_{T}-v_{T'}.$
    We obtain
    $$\int_X \exp(\alpha |u_T\circ \pi-u_{T'}\circ \pi|) \leq \left( \int_X \exp(2\alpha |v_{T}|)  \right)^\frac{1}{2} \left( \int_X \exp(2\alpha |v_{T'}|)  \right)^\frac{1}{2} $$
    and Skoda's uniform integrability theorem \cite[Theorem 8.11]{GZ} asserts that the right hand side is finite when $\epsilon$ is small enough.    
\end{proof}

\subsection{Diameter is zero}

\begin{prop}\label{prop8}
For $\mu^\mathbb{N}\otimes \nu$-almost every $(\omega, \theta) \in \Omega \times \partial \mathbb{H},$ we have 
$$\emph{diam}(\underline{f_{n}}^* \cdots \underline{f_1}^* \theta) \longrightarrow 0. $$
\end{prop}

\begin{proof}
    By Proposition \ref{prop7}, it suffices to show that $$I_{n}:=\frac{1}{\lambda_{n}} \int |u_T-u_{T'}| \circ f_1 \cdots f_{n} \longrightarrow 0$$
    for any $T,T' \in \mathcal{C}(\theta)$ where $\lambda_n:=\langle f_n^* \cdots f_1^*\theta , [L]\rangle$.
    We can write $$I_{n}=\int_0^{+\infty} \phi_{n}(t)\, dt$$
    with $\phi_n(t):=\text{vol}(|u_T-u_{T'}| \circ f_1 \cdots f_n \geq \lambda_n t).$
    Using Proposition \ref{prop5}, we have 
    $$\phi_{n}(t) \leq C_1^{-1}\text{vol}(|u_T-u_{T'}| \geq \lambda_{n}t)^\frac{1}{C_2\deg(f_1\cdots f_{n})}. $$
    Then, by Lemma \ref{lem16}, we have
    $$\phi_{n}(t) \leq C' \exp \left( -\frac{\alpha \lambda_{n}t}{C_2\deg(f_1 \cdots f_{n})}\right)$$
    for some $\alpha>0$ that will be chosen conveniently afterwards. By Lemma \ref{lem4}, we obtain $$\phi_{n}(t) \leq C' \exp(-\alpha C''t).$$
    
    Suppose there exists $t_0>0$ such that $\phi_{n}(t_0)$ does not converge to zero. Up to pass to a subsequence, we may assume $\phi_{n}(t_0)\geq \gamma >0$ for every $n$ and thus we have $$(C_1\gamma)^{C_2\deg(f_1 \cdots f_{n})} \leq \text{vol}(|u_T-u_{T'}| \geq \lambda_{n}t_0).$$
    As above, Lemmas \ref{lem4} and \ref{lem16} give
    $$(C_1\gamma)^{C_2\deg(f_1 \cdots f_{n})} \leq C'\exp (-\alpha C'' t_0 \deg(f_1 \cdots f_{n})).$$
    Notice that $C'$ depends on $\alpha$, but $C$ and $C''$ do not. We obtain a contradiction when $n \rightarrow \infty$ if $\alpha$ is large enough. 
    Thus, we have $\phi_{n} (t) \rightarrow 0$ for every $t>0$ and  $\lim I_{n}=0$ follows by dominated convergence.
\end{proof}

\begin{cor}\label{cor1}
For $\nu$-almost every $\theta \in \partial \mathbb{H}$, we have $\emph{diam}(\theta)=0$.    
\end{cor}

\begin{proof}
    Consider $A_\eta:=\{\text{diam}(\theta) \geq \eta \} $ and suppose that $A_\eta$ has positive measure $\delta >0$ for some $\eta>0$. Consider the dynamical system $F: \Omega \times \partial \mathbb{H} \rightarrow \Omega \times \partial \mathbb{H}$ defined by
    $F(\omega, \theta)=(\sigma(\omega), \underline{f_1}^*\theta).$
    The measure $\mu^\mathbb{N}\otimes \nu$ is $F$-invariant and ergodic. Birkhoff's ergodic Theorem applied to $\Omega \times A_\eta$ gives that 
    $$ \frac{1}{N}\# \{ 1\leq n \leq N, \, \underline{f_n}^* \dots \underline{f_1}^* \theta \in A_\eta  \} \longrightarrow \delta $$
    for $\mu^\mathbb{N}\otimes \nu$-almost every $(\omega, \theta)$. 
    This contradicts Proposition \ref{prop8}.
\end{proof}

\subsection{Further results}
Note that with the same ideas, we can prove the following theorem.
\begin{thm}
    Consider a birational map $f$ which satisfies the following property: there exists $C\geq 1$ such that the order of any of the base points of $f^n$ is bounded by $C$ for all $n \in \mathbb{Z}$.
    Then, for any $r\geq 2$, and $d_1, \dots, d_r \in \mathbb{N}^\times$, there exists a Zariski dense subset $V \subset \mathrm{Bir}_{d_1}(\mathbb{P}^2)\times \cdots \times \mathrm{Bir}_{d_r}(\mathbb{P}^2)$ such that for any probability measure $\mu$ supported on $\{f^{\pm 1}, g_1^{\pm 1},\dots, g_r^{\pm 1}\}$ with $(g_1, \dots, g_r) \in V$, the following hold.
    \begin{enumerate}
        \item For $\mu^\mathbb{N}$-almost every $\omega = (f_n)_{n \geq 1}\in \Omega$, there exists a positive closed $(1,1)$ current $T_\omega$ of mass one on $\mathbb{P}^2$ which does not charge curves, such that
        $$\frac{1}{\deg(f^n_\omega)} (f^n_\omega)^*S \longrightarrow T_\omega $$
        for every positive closed $(1,1)$ current $S$ of mass one on $\mathbb{P}^2$.
        \item Moreover, $T_\omega \neq T_{\omega'}$ for $\mu^\mathbb{N} \otimes \mu^\mathbb{N} $-almost every $(\omega,\omega') \in \Omega \times \Omega$.
    \end{enumerate}
\end{thm}

Indeed, in this case, we can prove that for $(g_1, \dots , g_r)$ generic, there exists a model $(X,\pi)$ on which all the base points of any element of the group $\Gamma=\langle f, g_1, \dots, g_r \rangle$ are defined on $X$. We can apply our proof where we choose $X$ as a base surface instead of $\mathbb{P}^2$.

This property is certainly not satisfied when $f$ is a loxodromic automorphism of the affine plane. Still, we have shown in \cite[Proposition 3.1]{N} that any limit current of normalized pullbacks of any Kähler form do not charge curves in any model when the measure $\mu $ is purely loxodromic. Moreover, using \cite[Proposition 2.7]{N}, we can prove volume estimates similar to Proposition \ref{prop5}. Thus, the method developed in the present paper applies also in this setting, and it gives an alternative way to prove \cite[Proposition 5.1]{N}.\\

Consider a rational map $f: \mathbb{P}^2 \dashrightarrow \mathbb{P}^2$ with small topological degree $\lambda_1>\lambda_2$ (or even with $\lambda_1^2>\lambda_2$). We can also define the action of $f$ by pullback on the space of Weil currents (and classes). In this setting, $f^*$ do not preserve the intersection form. Still, by a theorem of Boucksom--Favre--Jonsson \cite[Theorem 3.5]{BFJ}, we have convergence at the level of classes. The analogue of Proposition \ref{prop5} is also known. See for example \cite[Theorem 3.1]{Bay}. The hypothesis that $\deg(f^n) \leq C\lambda_1^n$ is implied by \cite[Main Theorem]{BFJ}. All other arguments of the proof can be adapted to the case of a rational map. Then, if we are able to prove that any limit of a subsequence of the normalized pullbacks of the Fubini-Study form do not charge curves in any model, we obtain the existence of the Green current.

\end{document}